\documentclass[makeidx]{amsart}

\usepackage{latexsym,amsthm,amssymb,amsmath,amsfonts,euscript,amstext}
\usepackage{epsfig}
\usepackage[dvips]{color}
\newtheorem{theorem}{Theorem}[section]
\newtheorem{lemma}[theorem]{Lemma}

\theoremstyle{definition}
\newtheorem{definition}[theorem]{Definition}

\newtheorem{corollary}[theorem]{Corollary}

\theoremstyle{remark}
\newtheorem{remark}[theorem]{Remark}
\numberwithin{equation}{section}

\newcommand{\abs}[1]{\lvert#1\rvert}
\newcommand{\norm}[1]{\lVert#1\rVert}

\begin{document}

\title[selection of steady finger by undercooling]{Rigorous Results in Existence and Selection of Saffman-Taylor Fingers by Kinetic Undercooling}

\author{Xuming Xie }

\address{Xuming Xie \newline
Department of Mathematics\\
Morgan State University\\
Baltimore, MD 21251, USA}
\email{xuming.xie@morgan.edu}

\begin{abstract}
The selection of Saffman-Taylor fingers by surface tension has been extensively investigated.
In this paper we are concerned with the existence and selection of steadily translating
symmetric finger solutions in a Hele-Shaw cell by small but non-zero
kinetic undercooling ($\epsilon^2 $).
We rigorously conclude that for relative finger width
$\lambda$ near one half,
symmetric finger solutions exist in the asymptotic limit
of undercooling $\epsilon^2 ~\rightarrow ~0$ if
the Stokes multiplier for a relatively simple nonlinear differential
equation is zero. This Stokes multiplier $S$
depends on the parameter
$\alpha \equiv \frac{2 \lambda -1}{(1-\lambda)}\epsilon^{-\frac{4}{3}} $ and
earlier calculations  have shown
this to be zero for  a discrete set of values of $\alpha$. While this result is similar to that obtained previously for Saffman-Taylor fingers by surface tension, the analysis for the problem with kinetic undercooling exhibits a number of subtleties as pointed out by Chapman and King (2003) [The selection of Saffman-Taylor fingers by kinetic undercooling, Journal of Engineering Mathematics 46, 1-32]. The main subtlety is the behavior of the Stokes lines at the finger tip, where the analysis is complicated by non-analyticity of coefficients in the governing equation.\\


\end{abstract}

\maketitle

\section{Introduction}
\label{intro}
\subsection{Background}
The problem of a less viscous fluid displacing a more viscous fluid in
a Hele-Shaw cell has been the subject of numerous studies since the 1950s.
In a seminal paper,
Saffman \& Taylor \cite{ST} found experimentally that
an unstable planar interface evolves
through finger competition to a steady
translating finger, with relative finger width $\lambda$ close
to one half.
Theoretical
calculations
\cite{Zhuravlev}, \cite{ST}
ignoring surface tension
revealed an one-parameter family of exact steady solutions,
parameterized by width
$\lambda$. When the experimentally determined $\lambda$ is
used, the theoretical shape (usually
referred to in the literature as the Saffman-Taylor finger) agreed well
with experiments for relatively
large displacement rates, or equivalently for small surface tension.
However, in the zero-surface-tension steady-state theory,
$\lambda $ remained undetermined in the (0, 1) interval. The selection
of $\lambda$ remained unresolved until the mid 1980s. Numerical calculations
\cite{McLean81}, \cite{Vandenbroeck83}, \cite{KesslerLevine85}, supported by
formal asymptotic calculations in the steady finger \cite{Combescotetal86},
\cite{Shraiman86}, \cite{Hong86},
\cite{Tanveer87},
\cite{Chapman1} and the closely-related
steady Hele-Shaw bubble problem
\cite{CombescotDombre88,34},
suggest that
a discrete family of solutions exist for which
the limiting shape, as surface tension tends to zero, approaches
the Saffman-Taylor finger with $\lambda = \frac{1}{2}$. Rigorous results were later obtained in \cite{Xie1, Xie2}.\\

The Hele-Shaw problem is similar to the Stefan problem in the context of melting or freezing. Besides surface tension, the physical effect most commonly incorporated in regularizing the ill-posed Stefan problem is kinetic undercooling \cite{35,36,37}, where the temperature on the moving interface is proportional to the normal velocity of the interface. The Stefan problem also arises in some other physical situations such as the diffusion of solvent through glassy polymers \cite{38,39} and the interfacial approximation of the streamer stage in the evolution of sparks and lightning \cite{40}.

For the Hele-Shaw problem, kinetic undercooling regularization first appeared in \cite{Romero, Saff86}.
Local existence of analytic solution was obtained in \cite{24,25} for the time dependent Hele-Shaw problem with kinetic undercooling. Using exponential asymptotics, Chapman and King \cite{Chapman2} analyzed the selection problem of determining  the discrete set of widths of a traveling finger for varying kinetic undercooling. Some numerical studies have been attempted in \cite{Dallaston,Gardiner}. A continuum of corner-free traveling fingers were obtained numerically in \cite{Dallaston} for any finger width above a critical value, while a discrete set of analytic fingers, as predicted in \cite{Chapman2}, were computed in \cite{Gardiner}. The physical connection between the kinetic undercooling effect and the action of the dynamic contact angle was established in \cite{Anjos}.

The aims of this paper are to establish some rigorous results in existence and selection of symmetric analytic solutions to the Hele-Shaw problem with sufficiently small kinetic undercooling.

\subsection{Mathematical Formulation}
We consider the problem of
a finger of inviscid fluid displacing a viscous fluid in a Hele-Shaw cell for small kinetic undercooling (Fig.\ref{fig:1}).
\begin{figure}
\begin{center}
\resizebox{0.75\textwidth}{!}
{\includegraphics{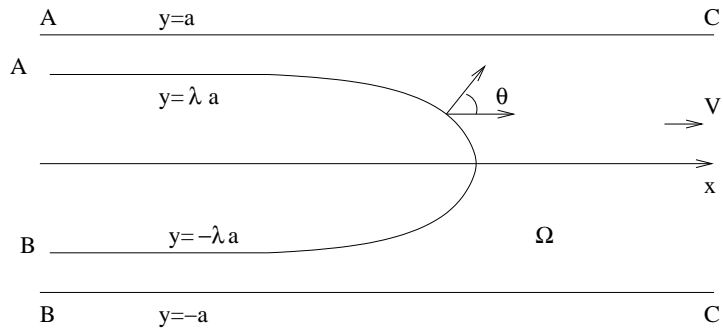}}
\end{center}
\caption{A translating finger in the Hele-Shaw cell}
\label{fig:1}
\end{figure}
 The gap-averaged
velocity $(u, v)$ in the $(x, y)$ plane in the region outside the
finger satisfies
\begin{equation}
\label{1.1}
(u, v) = - \frac{b^2}{12 \mu} \nabla p \equiv \nabla \phi,
\end{equation}
where $b$  is the gap width, $\mu$ the viscosity of the
viscous fluid and $p$ denotes the pressure.  Incompressibility of fluid flow implies zero divergence of
fluid velocity, implying that the velocity potential $\phi$ satisfies
\begin{equation}
\label{1.2}
\Delta \phi = 0  ~~~{\rm for} ~~(x, y) \in \Omega.
\end{equation}
The boundary conditions on the side walls are
\begin{equation}
\label{1.3}
v = 0 ~~{\rm when } ~~y = \pm{a}.
\end{equation}
The far field condition is
\begin{equation}
\label{1.4}
(u, v) = V {\hat x} + O(1) \text{ as } x\to +\infty,
\end{equation}
where ${\hat x}$ is a unit vector in the $x$-direction (along the Hele-Shaw
channel) and $V$ is the constant displacement rate of the fluid far away.
The kinematic condition for a steady finger is
\begin{equation}
\label{1.5}
\frac{\partial\phi}{\partial n}=U\cos\theta ,
\end{equation} 
where $\theta$ is the angle between the interface normal and the positive x -axis (see Fig. \ref{fig:1}) and $U$ is the speed of the finger.
The kinetic undercooling condition
\begin{equation}
\label{1.6}
\phi = c v_n=cU\cos\theta,
\end{equation}
where $c>0$ is the kinetic undercooling parameter. There are a number of different but equivalent formulations \cite{McLean81,Tanveer87,Chapman2,Tanveer2}, here we use the formulation which parallels that in Tanveer \cite{Tanveer2}.
We set the channel half-width $a=1$ and displacement rate $V=1$, which
corresponds to non-dimensionalizing all lengths by $a$ and
velocities by $V$ (consequently time is measured in units of $a/V$).

By integrating (\ref{1.2})
in the domain $\Omega$, the finger width $\lambda$ is related to $U$ as follows:
\begin{equation}
\lambda = \frac{1}{U} .
\end{equation}
In a frame moving with the steady symmetric finger, the condition (\ref{1.5})
transforms, without loss of generality, into
\begin{equation}
\psi =0,
\end{equation}
where $\psi$ is the stream function ( the harmonic conjugate of $\phi $ )
so that $W = \phi +i \psi $ is an analytic function of $z =x+iy$.
The nondimensional kinetic undercooling condition (\ref{1.6})
in the moving frame becomes
\begin{equation}
\label{1.9}
\phi + \frac{1}{\lambda} x = \frac{c}{\lambda}\cos\theta.
\end{equation}
On the side walls, (\ref{1.3}) implies that
\begin{equation}
\psi = \pm \left[\frac{1}{\lambda}-1\right] \text{ on }y=-1 \text{ and }
 y=1 \text{ respectively}
\end{equation}
while the far field condition as $z ~\rightarrow ~+\infty$
in the finger frame is
\begin{equation}
W\sim -\left[ \frac{1}{\lambda}-1\right]z+O(1).
\end{equation}

\begin{figure}
\begin{center}
\resizebox{0.75\textwidth}{!}
{\includegraphics{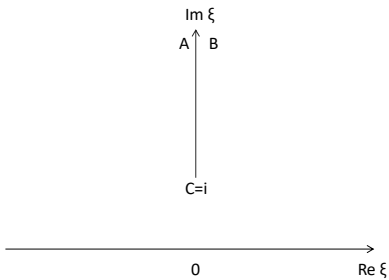}}
\end{center}
\caption{Upper half
$\xi$-plane with a cut on the imaginary axis from $\xi=i$ to $i \infty $}
\label{fig:2}
\end{figure}

We consider the conformal map $z(\xi)$ of the cut upper-half-$\xi$-plane,
as shown in Figure \ref{fig:2},
to the flow domain $\Omega$ in Figure \ref{fig:1}.
The real $\xi$ axis corresponds to the finger boundary,
with $\xi=-\infty,+\infty$ corresponding to the finger tails at
$z=-\infty+\lambda i,z= -\infty -\lambda i$ respectively. The two
sides of the cut correspond to the two side walls respectively.
It is easily seen that the complex potential $W(\xi)$ is given by:
\begin{equation}
W(\xi)=\frac{1-\lambda}{\pi\lambda}\ln~ (\xi^2+1) + constant.
\end{equation}

We define $F(\xi)$ so that
\begin{equation}
\label{1.13}
z(\xi)=-\frac{1}{\pi}\ln~(\xi-i)-\frac{1-2\lambda}{\pi}\ln~(\xi+i)-\lambda i
-\frac{2-2\lambda}{\pi} F(\xi).
\end{equation}
It follows that $F$, as defined above, is analytic in the entire upper half
$\xi$-plane \cite{Tanveer2}.
The kinetic undercooling condition (\ref{1.9}) translates into requiring that on the real
$\xi$ axis:
\begin{equation}
\label{1.14}
\text{Re } F= -\epsilon^2 \frac{\text{Im }(F'+H)}{|F'+H|},
\end{equation}
where
\begin{equation}
\label{1.15}
H(\xi)=\frac{\xi+\gamma i}{\xi ^2+1},\quad \text { with } \gamma =\frac{\lambda}{1-\lambda},~\epsilon^2=\frac{c\pi }{2\lambda(1-\lambda)}.
\end{equation}
For zero kinetic undercooling $\epsilon =0$, it follows that $F=0$; this
corresponds to what is usually referred to in the literature as the
Saffman-Taylor solutions \cite{Zhuravlev,ST}.
This form
a family of exact solutions for symmetric fingers
with arbitrary width of finger $\lambda\in(0,~1)$ .
The Saffman-Taylor solutions, in our formulation,
correspond to the conformal map
\begin{equation}
\label{1.16}
z_0(\xi)=-\frac{1}{\pi}\ln~(\xi-i)-\frac{1-2\lambda}{\pi}\ln~(\xi+i)-\lambda i.
\end{equation}
This is easily seen to be univalent since the boundary correspondence
is one to one.
In particular, the finger shape can be explicitly described
by Re $z$ as a function of Im $z$.

For non-zero kinetic undercooling, no exact solutions exist.
  Numerical and asymptotic
work \cite{Chapman2, Dallaston, Gardiner} suggested that solutions
exist  for infinitely many isolated values of finger width which are larger than $\frac{1}{2}$. The aim of this paper is to obtain rigorous results for this selection problem, in particular the existence of a symmetric and analytic solution for sufficient small undercooling constant $c$.

\subsection{Notations}
\label{sec:1.2}
\begin{definition}\label{def:1.1}
Let $\mathcal{R}$ be an open connected (see Figure \ref{fig:3}) region on complex $\xi$ plane bounded by lines $r_u$ and $r_l$ defined as follows:
$$r_u=r_{u,1}\cup r_{u,2}\cup r_{u,3}\cup r_{u,4}\cup r_{u,5}\cup r_{u,6}$$
$$r_l=\{ \xi :\xi =-bi+re^{-i(\varphi_0+\mu)}\}
\cup\{ \xi :\xi =-bi+re^{i(\pi+\varphi_0+\mu)}\}$$
where $ 0 < b < \min ~(1, \gamma)$,  $ ~0<\varphi_0,~\mu~<~\frac{\pi}{2}$
with $\varphi_0+\mu < \frac{\pi}{2}$ and
$$r_{u,1}=\{ \xi :\xi = \nu_1 i-R+re^{i(\pi-\varphi_0)},0\leq r<\infty\}~,~$$
$$r_{u,2}=\{ \xi : \text{ Im }\xi = \nu_1, -R \le \text{Re }\xi \le -\nu_1\},$$
$$ r_{u,3}=\{ \xi:\xi=re^{3\pi i/4} ,0 \le r\leq\sqrt{2}\nu_1 \} ,~$$
$$r_{u,4}=\{ \xi:\xi=r e^{\pi i/4} \},0 \le r\leq \sqrt{2}\nu_1\}, $$
$$r_{u,5}=\{ \xi : \text{ Im }\xi = \nu_1, \nu_1 \le \text{ Re }\xi \le R \}~,~$$
$$r_{u,6}=\{ \xi :\xi =\nu_1 i+R+re^{i\varphi_0},0\leq r<\infty\},$$
\end{definition}
where $R>0$ is large enough and $0<\nu_1$ is small enough so that Lemma 6.3 -  Lemma 6.8 in the appendix hold.\\

\begin{figure}
\begin{center}
\resizebox{0.75\textwidth}{!}
{\includegraphics{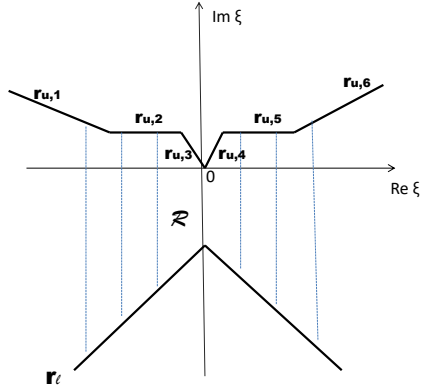}}
\end{center}
\caption{Region $\mathcal{R}$ in the complex $\xi$ plane given in Definition 1.1 }
\label{fig:3}
\end{figure}

\begin{remark}
\label{rem:1.1}
$1-b$, and $\varphi_0$
are chosen independently of $\epsilon$. There are some restrictions
imposed on their values in order that certain lemmas
in the appendix are ensured.

\end{remark}

\begin{definition}\label{def:1.2}
$\mathcal{R}^-=\mathcal{R}\cap \{\text{ Re }\xi~<~ 0\}~;~
\mathcal{R}^+=\mathcal{R}\cap \{\text{ Re }\xi~>~ 0\} $
\end{definition}
Let $0 <\tau <1$ be independent of $\epsilon$, $\nu>0$ be a small number, and $B_\nu=\{ \xi:|\xi|<\nu\}$.
We introduce spaces of functions:
\begin{definition}\label{def:1.3}
 Define
\begin{equation*}
\begin{split}
\mathbf{A}^-_0&=\{F:
F(\xi)~\text{analytic in } \mathcal{R}^-\text{~and continuous in } \overline{\mathcal{R}^-},\\
&~~~\text{ with }\underset{\xi\in\overline{\mathcal{R}^-}}{\sup}~\left\vert(\xi-2i)^{\tau}F(\xi)\right\vert
<\infty \\
&~~~ \text{and } \underset{\xi\in\overline{\mathcal{R}^-}}{\sup}~\left\vert\frac{F'(\xi)-F'(0)}{\sqrt{\xi}}\right\vert <\infty\}
\end{split}
\end{equation*}
$$\norm{F}_0^-:=\underset{\xi\in\overline{\mathcal{R}^-}}{\sup}~\abs{(\xi-2i)^{\tau}F(\xi)}+\epsilon^2 \underset{\xi\in\overline{\mathcal{R}^-\cap B_\nu} }{\sup}~\left\vert\frac{F'(\xi)-F'(0)}{\sqrt{\xi}}\right\vert +|F'(0)|. $$
\end{definition}

\begin{remark}
\label{rem:1.2}
$\mathbf{A}^-_0$ is a Banach space. If $F\in\mathbf{A}_0^-$,
then $F$ satisfies property:
\begin{equation}
\label{1.17}
F(\xi)\sim O(\xi^{-\tau}),~\text{ as } |\xi|\to\infty,~\xi\in
\mathcal{R}^-.
\end{equation}
In Definition 1.4, replacing $\mathcal{R}^-$ with $\mathcal{R}$ we can define $\mathbf{A}\equiv\mathbf{A}_0$ with norm
$$\norm{F}_0:=\underset{\xi\in\overline{\mathcal{R}}}{\sup}~\abs{(\xi-2i)^{\tau}F(\xi)}+\epsilon^2 \underset{\xi\in\overline{\mathcal{R}\cap B_\nu} }{\sup}~\left\vert\frac{F'(\xi)-F'(0)}{\sqrt{\xi}}\right\vert +|F'(0)| .$$
\end{remark}

\begin{remark}
Locally in a neighborhood of $\xi=0$, $\mathbf{A}_0$ contains functions which are analytic at $\xi=0$, it also contains functions $F(\xi)$ which are NOT analytic at
$\xi=0$ such as $F(\xi)=C\xi^{k+3/2}, k=0,1,2, \cdots $ with a branch cut on upper half plane. If $F(\xi)\in \mathbf{A}_0$ is analytic in the upper half plane, then $F(\xi)$ has to be analytic at $\xi=0$.
\end{remark}
\begin{definition}\label{def:1.5}
Let $\mathcal{D}$ be any connected set
in the complex $\xi$-plane; we introduce norms:
$\norm{F}_{k,\mathcal{D}}:=\underset{\xi\in\mathcal{D}}{\sup}~\abs{(\xi-2i)^{k+\tau}F(\xi)},\quad k=0,1,2$.
\end{definition}
\begin{definition}\label{def:1.6}
Let $\delta >0$ be a constant, define space:
$$\mathbf{A}^-_{0,\delta}=\{f:f\in\mathbf{A}_0^-,~\norm{f}^-_0\leq\delta \}. $$
$\mathbf{A}_{0,\delta}$ can be defined similarly.
\end{definition}

\begin{remark}
\label{thm:1.1.2}
For $\lambda$ in a compact subset of (0, 1) and  for sufficiently
small enough $\delta$, if
$F \in \mathbf{A}_{0, \delta}$, then the
mapping $z(\xi)$ given by (\ref{1.13}) is univalent. See Theorem 1.5 in \cite{Xie1}.
\end{remark}

The problem of determining a smooth steady symmetric finger
is equivalent to finding function $F$ analytic 
in the upper-half
$\xi$ plane, which is $\mathbf{C^1}$ in its closure, {\it i.e.}
continuous derivative  on Im $\xi=0$
and is required to satisfy the following conditions:\\
{\bf Condition (i)}:
$F(\xi)$ satisfies (\ref{1.14}) on the real $\xi$ axis.

\noindent{\bf Condition (ii)}:
\begin{equation}
F(\xi) \to 0 \quad \text{ as }\xi\to\pm\infty.
\end{equation}

\noindent{\bf Condition (iii) ( symmetry condition )}:
\begin{equation}
\label{1.19}
\text{Re }F(-\xi)=\text{Re }F(\xi)~,~\text{Im }F(-\xi)=-\text{Im }F(\xi)
\quad \text{for }~\xi ~\text{real}.
\end{equation}

\begin{remark}
\label{rem:1.5}
If $F\in \mathbf{A}_0$ is analytic in the upper half plane, and satisfies symmetry condition (iii) on the
real axis, then it follows from successive Taylor expansions of $F(\xi)$
on the imaginary $\xi$ axis segment $-b ~<~\text{Im  }\xi $,
starting at $\xi = 0$ that
Im $F(\xi)=0$.
From the Schwartz reflection principle, $ F(\xi) = \left [ F(-\xi^*) \right ]^* $,
so $\norm{F}_0^-=\norm{F}^+_0$.
Conversely, if $F\in \mathbf{A}_0^-$ and
satisfies Im $ F (\xi) = 0$
for the imaginary $\xi$ axis
segment $-b ~<~\text{Im  }\xi ~<~0$,
then $F(\xi) = \left [ F(-\xi^*) \right ]^* $ extends $F$ to the right half of $\mathcal{R}$
and $\norm{F}_0^-=\norm{F}_0$ and (1.19) is then automatically satisfied.
\end{remark}

\noindent{\bf Finger problem:} {\it The problem tackled in this paper will be to find
function $F$ analytic in $\{\text{Im }\xi>0\}\cup\mathcal{R}$
so that $F\in\mathbf{A}_{0,\delta}$ for some $\tau$ fixed in (0, 1), $\delta$
 small, so that conditions
(i) and (iii) on the real axis are satisfied.}

\subsection{Main Results}
\label{sec:1.3}
Similar to the finger problem with surface tension \cite{Combescotetal86, Combescotetal88, Tanveer87, Xie1}, the formal strategy of calculation of finger width involves analytic continuation of equation in an inner neighborhood of turning points in the complex plane and ignoring integral contribution and other terms that are formally small. As in Chapman and King \cite{Chapman2}, the problem of determining a smooth steady finger is reduced to determining eigenvalues $\alpha$ so that $G(y)$ is a solution to
\begin{equation}
\label{1.20}
\frac{dG}{dy} -\frac{1}{yG^{2}}=-y-\frac{\alpha}{2^{1/3}y}
\end{equation}
satisfying the condition
\begin{equation}\label{1.21}
yG(y)\to 1 \quad \text{ as } y\to \infty \text{ and } \arg y\in[0, \pi/4),
\end{equation}
and in addition satisfying:
\begin{equation}\label{1.22}
\text{Im  }G=0,\text{ for large enough $y$ on the positive real axis.}
\end{equation}
The finger width $\lambda$ is related to $\alpha$ through
\begin{equation}\label{1.23}
\alpha=\frac{2\lambda-1}{1-\lambda}\epsilon^{-4/3}.
\end{equation}
We introduce additional change in variables:
\begin{equation}\label{1.24}
\eta=\frac{2}{3}y^{3}~,~\psi(\eta)=1-yG(y),
\end{equation}
then (\ref{1.20}) becomes:
\begin{equation}
\label{1.25}
\frac{d\psi}{d\eta}+\psi=-\frac{1}{3\eta}-\frac{1}{3\eta}\psi
+\frac{\alpha}{6^{2/3}\eta^{2/3}}+\frac{1}{2}\sum_{n=2}^{\infty}(-1)^n(n+1)\psi^n
\end{equation}
and matching condition:
\begin{equation}
\label{1.26}
\lim_{\eta\to\infty,~\arg \eta\in \left(0,\frac{3\pi}{4}\right)}\psi(\eta,\alpha) = 0.
\end{equation}
One of the  theorems proved in this paper is
\begin{theorem}\label{thm:1.8}
(1) There exists sufficiently large enough $\rho_0 >0$ such that (\ref{1.25})
with condition (\ref{1.26})
has unique analytic solution $\psi_0(\eta,\alpha)$ in the region $|\eta|>\rho_0,
~\arg\eta\in [0,\frac{3\pi}{4}]$.\\
(2) Further on the real $\eta$ axis as $\eta\to \infty$
\begin{equation}
\mathrm{ Im  }~\psi_0(\eta,\alpha)\sim \tilde S(\alpha)e^{-\eta}.
\end{equation}
(3) Also $\mathrm{ Im }~\psi_0(\eta,\alpha)=0 $ for real $\eta$ and $\eta>\rho_0$
if and only if $\tilde S(\alpha)=0$.
\end{theorem}
We will not compute the Stokes constants $\tilde S(\alpha)$ in this paper,
Chapman and King's
numerical computation \cite{Chapman2} indicates that there exist a discrete set
$\{\alpha_n\}$ so that $\tilde S(\alpha_n)=0,\tilde S'(\alpha_n)\neq
0$.
It is to be noted that
the theory of exponential asymptotics for general nonlinear
ordinary differential equations \cite{Costin} makes it possible to
rigorously confirm these calculations to within small error bounds.
 This analysis will be carried out for this problem in a future work.

The primary result of this paper for the finger problem is the following:
\begin{theorem}\label{thm:1.9}
In a range $\frac{1}{2} \le \lambda \le \lambda_m$,
$\delta$ and $\lambda_m - \frac{1}{2}$ small
(but independent of $\epsilon$)
so that (\ref{2.24}) holds,
the following statement hold for all sufficiently small $\epsilon$:\\
 For each $\beta_0$
for which the Stokes constant $\tilde S(\beta_0)=0$
in Theorem \ref{thm:1.8}, if ${\tilde S}^\prime (\beta_0) ~\ne ~0$,
there exists  an analytic function $\beta (\epsilon^{2/3})$
with $\lim_{\epsilon\to 0}\beta (\epsilon^{2/3})=\beta_0$ so that if
\begin{equation}
\label{1.28}
\frac{2\lambda -1}{1-\lambda }=\epsilon^{4/3} \beta (\epsilon^{2/3}),
\end{equation}
then there exists a solution of the Finger problem
$F\in \mathbf{A}_{0,\delta}$.
Hence for small $\epsilon$,
\begin{equation}
\label{1.29}
\frac{2\lambda -1}{1-\lambda }=\epsilon^{4/3} \left( \beta_0
+\beta_1\epsilon^{2/3}+\beta_2\epsilon^{4/3}+\cdots\right).
\end{equation}

\end{theorem}
The proof of this theorem is  given at the end of \S4, after many
preliminary results.
Our solution strategy consists of two steps:

(a) Relaxing the symmetry condition $F(\xi)=[F(-\xi^*)]^*$ on the imaginary
axis interval $-b<\text{Im }\xi<0$, ({\it i.e.} Im $F=0$ is relaxed
on that Im $\xi$-axis segment), we prove the existence
of solutions to an appropriate problem in the half strip $\mathcal{R}^-$
 for any $\lambda $ belonging to a compact subset of
$(0,1)$, for all sufficiently small $\epsilon$.
There is no restriction on $\lambda$ otherwise.

(b) The symmetry condition is then invoked to determine
a restriction on $\lambda$
that will guarantee existence of solution to the Finger problem.
In this part, we restrict our analysis to
$\lambda \in [\frac{1}{2} , \lambda_m]$.

In Section 2, we prove equivalence
of the finger problem to a set of two problems ({\it Problem 1 and Problem 2})  in a complex strip domain. Problem 1 is to solve an integro-differential
equation for $F$ in a Banach space $\mathbf{A}_0$. By deforming the contour
of integration for the integral term in Problem 1, we obtain Problem 2.
By relaxing symmetry condition on an Im $\xi$ axis segment,  we derive
the Half Problem in the left half strip $\mathcal{R}^-$.

In Section 3,
by  constructing a normal sequence, we prove the existence
 of  solutions
to the Half
 Problem for
$\lambda$ in a compact subset of (0, 1) for all sufficiently small $\epsilon$.
In Section 4, we carry out step (b) in our solution strategy.
By introducing suitably scaled dependent and independent variables
in a neighborhood of a turning point,
we formulate the inner problem.
For the leading order inner-equation,
the form of exponentially small terms are obtained and
Theorem \ref{thm:1.8} is proved.
For the full problem, using the implicit function theorem,
it is argued that for small $\epsilon$, and $\beta_0$ so that $\tilde S(\beta_0)=0$ and $\tilde S'(\beta_0)\ne 0$,
there exists a analytic functions $\beta (\epsilon^{2/3})$
so that if $\alpha=\beta(\epsilon^{2/3})$,
then Im $F=0$ on $\{ \text{ Re }\xi = 0 \} \cap \mathcal{R}$.
This implies that symmetry condition (\ref{1.19}) is satisfied;
hence Theorem \ref{thm:1.9} follows.

While the main result and strategy of proof are similar to that in \cite{Xie1} for Saffman-Taylor fingers by surface tension, the analysis for the problem with kinetic undercooling exhibits a number of subtleties as pointed out by Chapman and King \cite{Chapman2}. The main subtlety is the behavior of Stokes lines at the finger tip, where the analysis is complicated by non-analyticity of a coefficient in equation (\ref{3.12}) in section 3.

\section{Formulation of Equivalent Problems}
\label{sec:2}

In this particular section,
we are going to formulate Problem 1 and Problem 2
which will be proved to
be equivalent to the Finger Problem defined in Section 1.
We then formulate a half problem in terms of an integro-differential
equation in $\mathcal{R}^-$, but relax the symmetry condition
Im $F = 0$ on the imaginary $\xi$-axis segment $(-b i, 0)$, which
follows from (\ref{1.19}) (see Remark \ref{rem:1.5}).
Problem 2, unlike Problem 1, involves
nonlocal quantities ($I_2 $ for instance)
with integration paths outside domain $\mathcal{R}$.
Hence, when symmetry is relaxed for $F$ in the half problem,
singularities at the origin from
nonlocal contributions can be estimated conveniently in terms of $F$ from
the domain $\mathcal{R}^-$ . In this section, as well as
the next, we will
restrict $\lambda$ to a compact subset of (0, 1) so that all
the constants appearing in all estimates are independent of $\lambda$.

\subsection{Formulation of Problem 1}
\begin{definition}
\label{def:2.1.5}
Define
\begin{equation}
\label{2.1}
\bar{H}(\xi)=[H(\xi ^*)]^*=\frac{\xi-\gamma i}{\xi^2+1},
\end{equation}
\begin{equation}
\label{2.2}
\bar{F}(\xi)=[F(\xi^*)]^*.
\end{equation}
\end{definition}
\begin{remark}
\label{rem:2.2}
If $F$ is analytic in domain $\mathcal{D}$
containing the real axis
with property (\ref{1.17}),
then $\bar{F}$ is analytic in $\mathcal{D^*}$ and
$\bar{F}(\xi)=F^*(\xi)$ for $\xi$ real and
$\bar{F}(\xi)\sim O(\xi^{-\tau})$, as $|\xi|\to
\infty,\xi\in \mathcal{D^*}$;
~
$\bar{F'}(\xi)\sim O(\xi^{-1-\tau})$,
$\bar{F''}(\xi)\sim O(\xi^{-2-\tau})$ as $|\xi|\to
\infty,\xi\in \mathcal{D'^*}$,
where $\mathcal{D'}$ is any angular subset of $\mathcal{D}$ and superscript
${}^*$ denotes conjugate domain obtained by reflecting about the real
axis.
\end{remark}
\begin{definition}\label{def:2.2}
Let $\mathcal{D}$ be a connected set; for any two functions $f,g$ with  derivative existing in $\mathcal{D}$ and small enough
$\norm{f'}_{1,\mathcal{D}}$ and $\norm{g'}_{1,\mathcal{D}}$
so that $f'+H\neq 0,g'+\bar H\neq 0$ in $\mathcal{D}$, we
define operator $G$ so that
\begin{equation}
\label{2.3}
G(f,g)[t]:=\frac{1}{(f'(t)+H(t))^{1/2}(g'(t)+\bar{H}(t))^{1/2}}\times
\left[(f'(t)+H(t))
-(g'(t)+\bar{H}(t))\right].
\end{equation}
\end{definition}

\begin{remark}
\label{rem:2.3}
If $F\in\mathbf{A}$ and $F'+H\neq 0$ in $\mathcal{R}$, then $G(F,\bar F)(t)$ is analytic in $\mathcal{R}\cap\mathcal{R}^*$,
since in that case $\bar F'+\bar H\neq 0$ in $\mathcal{R}^*$.
\end{remark}
\begin{lemma}\label{lem:2.4}
Let $\mathcal{D}$ and
$f,g$ be as in definition \ref{def:2.2}.
If each of dist $(\mathcal{D},-i)$, dist $(\mathcal{D}, i)~$,
dist $(\mathcal{D},-\gamma i)$, and dist $(\mathcal{D}, \gamma i)$ are
greater than 0 and independent of $\epsilon$,
as $\epsilon\to 0$, then
we have for small enough
$\norm{f'}_{1,\mathcal{D}},\norm{g'}_{1,\mathcal{D}}$,
\begin{equation}
\label{2.4}
\norm{G(f,g)}_{0,\mathcal{D}}\leq \frac{C(1+
\norm{g'}_{1,\mathcal{D}}
+\norm{f'}_{1,\mathcal{D}})}
{( K_1-\norm{f'}_{1,\mathcal{D}})(K_1-\norm{g'}_{1,\mathcal{D}})},
\end{equation}
where
\begin{equation}
\label{2.5}
\begin{split}
0<H_m&\leq\underset{\mathcal{D}}{\inf}~\{|(t-2i)H(t)|,|(t-2i)\bar H(t)|\},\\
K&\geq\underset{\mathcal{D}}{\sup} ~|t-2i|^{-\tau}>0,\\
K_1&=\frac{H_m}{K}>0.
\end{split}
\end{equation}
Constants $C$,
$K$ and $K_1$ are independent of $\epsilon$ and $\lambda$, since
$\lambda$ is in a compact subset of $(0, 1)$.
\end{lemma}
\begin{proof}
Without ambiguity, the norms $\norm{\cdot}$ denoted in this proof
refer to $\norm{\cdot}_{.,\mathcal{D}}$,
where $\sup$ is over the domain $\mathcal{D}$.

Using (\ref{1.15}), (\ref{2.1}):
\begin{gather*}
\sup\{|t-2i|^2|\bar{H}(t)-H(t)|\}\leq C,\\
\sup|(t-2i)H|\leq C,
~\sup|(t-2i)\bar H|\leq C,\\
\end{gather*}
where $C$ is made independent of $\epsilon$ and $\lambda$ for
$\lambda$ in a fixed compact subset of (0, 1).
We also have
\begin{equation}
\label{2.6}
\begin{split}
|(f'(t)+H(t))|&\geq \left[ H_{m}|t-2i|^{-1}-\norm{f'}_1|t-2i|^{-1-\tau}\right]\\
&\geq (H_m-K\norm{f'}_1)|t-2i|^{-1}\\
&\geq C(K_1-\norm{f'}_1)|t-2i|^{-1},\\
|(g'(t)+\bar H(t))|&\geq \left[ H_{m}|t-2i|^{-1}-\norm{g'}_1|t-2i|^{-1-\tau}\right]\\
&\geq (H_m-K\norm{g'}_1)|t-2i|^{-1}\\
&\geq C(K_1-\norm{g'}_1)|t-2i|^{-1},
\end{split}
\end{equation}
where $C$ is made independent of $\epsilon$ and $\lambda$.
(\ref{2.4}) follows immediately from (\ref{2.6}) .
\end{proof}
\begin{lemma}\label{lem:2.5}
If $F \in \mathbf{A}$, then $G(F,\bar F)(t)=O(t^{-\tau}),$ as $|t|\to\infty,t$
in any angular subdomain of $\mathcal{R}\cap\mathcal{R}^*$.
\end{lemma}
\begin{proof}
From Remark \ref{rem:1.2},
Remark \ref{rem:2.2} and Lemma \ref{lem:2.4}, with $\mathcal{D}$
being an angular subdomain of $\mathcal{R} \cap \mathcal{R^*}$,
the proof follows.
\end{proof}

\begin{remark}
\label{rem:2.4}
Let $H_m,K,K_1$ be as in (\ref{2.5})
of Lemma \ref{lem:2.4} with $\mathcal{D}$ being an angular subdomain of $\mathcal{R} \cap \mathcal{R^*}$,
then $H_m,K,K_1$ can be chosen
to be independent of $\epsilon$ and $\lambda$
since $\lambda$ is in a fixed compact subset of (0, 1).
If $\norm{F'}_{1,\mathcal{D}}<\frac{K_1}{2}$  then,
\begin{equation}
\label{2.7}
|\xi-2i||F'+H|\geq (H_m-|\xi-2i|^{-\tau}\norm{F'}_1)\geq C(K_1-\norm{F'}_1)>C\frac{K_1}{2}>0 ;
\end{equation}
so $F'+H\ne 0$ holds in $\mathcal{D}$.
\end{remark}

\begin{remark}
\label{rem:2.5}
Let $\norm{F'}_{1,\mathcal{D}}<\frac{K_1}{2}$ , by
Remark \ref{rem:2.4},
we have $\bar F'+\bar H\neq 0$
in $\mathcal{R}^*$, so $G(F,\bar F)$ is
analytic in $\mathcal{D}$ which is  an angular subdomain of $\mathcal{R}\cap\mathcal{R}^*$.
\end{remark}
\begin{definition}\label{def:2.3}
Let $F\in\mathbf{A}$,$\norm{F'}_{1,\mathcal{D}}<\frac{K_1}{2}$, $\mathcal{D}$ is  an angular subdomain of $\mathcal{R}\cap\mathcal{R}^*$ and $(-\infty,\infty)\subset\mathcal{D}$,
define operator $I$ so that
\begin{equation}\label{2.8}
I(F)[\xi]=
-\frac{1}{2\pi }\int_{-\infty}^{\infty}
\frac{G(F,\bar{F})(t)dt}{t-\xi} \quad \text{ for Im }\xi ~<~0.
\end{equation}
\end{definition}
\begin{lemma}\label{lem:2.7}
Let  $F\in\mathbf{A}$ be analytic in  $\mathrm{Im }~\xi>0$ as well. If $F (\xi)$
satisfies equation (1.14) on the real axis,
then $F$ satisfies  for $\xi\in \{~\mathrm{ Im }~\xi<0\}$
\begin{equation}\label{2.9}
F'(\xi)=\frac{1}{2\epsilon^4}\big\{-\big[ G_2(F, I)^2(\bar F'+\bar H)+2\epsilon^4 G_1(\bar F)\big] +(\bar F'+\bar H)G_2(F,I)\sqrt{G_2(F,I)^2-4\epsilon^4}\big\}
\end{equation}
where
\begin{equation}\label{2.10}
G_1(\bar{F}')=-\bar{F}' +H- \bar{H},
\end{equation}
\begin{equation}\label{2.11}
G_2(F,I)=(F+\epsilon^2I(G)),
\end{equation}
 and the principal branch is chosen for the square root in (\ref{2.9}).
\end{lemma}
\begin{proof}
Since $F$
is analytic in
$\{ \text{ Im } \xi ~>~0\}$
and satisfies equation (\ref{1.14}),
by using the Poisson Formula, we have for Im ~$\xi ~>~0$:
\begin{equation}
\label{2.12}
\begin{split}
F(\xi)&=-\frac{\epsilon^2}{\pi i}\int_{-\infty}^{\infty}
\frac{dt}{(t-\xi)}\frac{1}{\abs{F'(t)+H(t)}}\text{ Im }
\left[F'(t)+H(t)\right]\\
&=\frac{\epsilon^2}{2\pi }\int_{-\infty}^{\infty}\frac{G(F,\bar{F})(t)dt}{t-\xi}.
\end{split}
\end{equation}
Using the Plemelj Formula (see for eg.
Carrier, Krook \& Pearson \cite{1}),
we analytically extend above equation to the lower half plane to obtain
\begin{equation}\label{2.13}
F(\xi)=-\epsilon^2 I(\xi)
+i\epsilon^2\frac{(F'(\xi)+H(\xi))-(\bar{F}'(\xi)+\bar{H}(\xi))}{(F'(\xi)+H(\xi))^{1/2}(\bar{F}'(\xi)+\bar{H}(\xi))^{1/2}},
~~\text{for ~Im }\xi<0.
\end{equation}
Squaring the above equation to obtain
\begin{equation}\label{2.14}
\epsilon^4 (F')^2 + G_3(F, \bar F, I) F' + G_4(F, \bar F, I)=0,
\end{equation}
where
\begin{equation}\label{2.15}
G_4(F,\bar{F},I)=H(F+\epsilon^2I(G))^2(\bar{F}' + \bar{H})+\epsilon^4(G_1(\bar{F}))^2 ,
\end{equation}
\begin{equation}\label{2.16}
G_3(F,I,\bar{F})=(\bar{F}' + \bar{H})(F+\epsilon^2I(G))^2+2\epsilon^4G_1(\bar{F}) ,
\end{equation}
Solving $F'$ from above leading to (\ref{2.9}), where the principal branch of the square root is chosen to be consistent with (\ref{2.13}).
\end{proof}

\noindent\begin{definition}\label{def:2.4}

Define two rays:
\[
r^+_0=\{\xi:\xi=\rho e^{i(\varphi_0+\frac{1}{2}\mu)},
0<\rho<\infty\},
\]
\[
r^-_0=\{\xi:\xi=\rho e^{i(\pi -\varphi_0-\frac{1}{2}\mu)},
0<\rho<\infty\}.
\]
Let $r_0$ be the directed contour along $r_0^-\cup r_0^+$ from
left to right (See Fig. 5).
\end{definition}
\begin{definition}\label{def:2.5}
If $F \in \mathbf{A}$, define
$F'_-(\xi)$ for $\xi$ below $r_0$:
\begin{equation}
\label{2.17}
F'_- (\xi)=-\frac{1}{2\pi i}\int_{r_0}
\frac{\bar{F}'(t)}{t-\xi}dt.
\end{equation}
\end{definition}
\begin{remark}
\label{rem:2.6}
Since $r_0\subset\mathcal{R}^*$ and
$\bar{F}$ satisfies (\ref{1.17}) in $\mathcal{R}^*$, the above integral
is well defined.
It is obvious that $F_-(\xi)$ is analytic below $r_0$.
Also, if $F \in \mathbf{A}_0^-$ only and the
relation $\left [F (-t^*) \right ]^* = F(t) $
were invoked to define $F$ in $\mathcal{R}^+$,
then it is possible to use
the symmetry between contours $r_0^+$ and $r_0^-$
to write
\begin{equation}
\label{2.18}
\begin{split}
F'_-(\xi)
&=-\frac{1}{2\pi i}\int_{r_0}
\frac{\bar{F}'(t)-\bar{F}'(0)}{t-\xi}dt +\bar{F}'(0)\\
&=
-\frac{1}{2\pi i} \left [\int_{r_0^-}
\frac{(\bar{F}'(t)-\bar{F}'(0)) dt}{t-\xi} + \int_{r_0^+}
\frac{([\bar{F}'(-t^*)]^*-\bar{F}'(0)) dt }{t-\xi} \right ] +\bar{F}'(0)\\
&=~-\frac{1}{2\pi i} \int_{r_0^-}
\left [\frac{(\bar{F}'(t)-\bar{F}'(0)) dt}{t-\xi} -
\frac{([\bar{F}'(t)]^* -[\bar{F}'(0)]^*)(dt)^* }{t^*+\xi} \right ]+\bar{F}'(0).
\end{split}
\end{equation}
This alternate expression is equivalent to (\ref{2.17})
when symmetry condition
Im $F = 0$ on $\{\text{ Re ~}\xi =0 \} \cap \mathcal{R}$ holds; however,
(\ref{2.18}) defines an analytic function $F_-$ below $r_0$
even when symmetry condition is relaxed, as it is for the Half
Problem  later in \S 2.
We also notice that $F'_-$, as defined by
(\ref{2.18}), satisfies symmetry
condition Re $F'_-=0$ on $ \{ \text{ Re }\xi=0\}\cap\mathcal{R}$ even when $F \in \mathbf{A}^-$ does not satisfy Im $F=0$ on $ \{\text{ Re~}\xi=0\}\cap\mathcal{R}$ .
\end{remark}
\begin{lemma}\label{lem:2.10}
If $F \in \mathbf{A}$ and $F$ is also analytic in $\mathrm{Im} ~\xi ~>~0$,
then $\bar{F}'(\xi)=F_-' (\xi), \text{ for }\xi\in \mathcal{R}$.
\end{lemma}
\begin{proof}
Since $F$ is analytic  in $\mathcal{R}~\cup~\{\text{ Im }\xi>0\}$,
then $\bar{F} $
is analytic in $\mathcal{R^*}~\cup~\{\text{ Im }\xi<0\}$.
We use property (\ref{1.17}) and the Cauchy Integral Formula to obtain
\begin{equation*}
F_- '(\xi)=-\frac{1}{2\pi i}\int_{r_0}
\frac{\bar{F}'(t)}{t-\xi}dt=\bar{F}' (\xi),\text{ for }\xi\in \mathcal{R}.
\end{equation*}
\end{proof}
\begin{lemma}\label{lem:2.11}
Let $\Gamma =\{t,t=~\rho e^{i\varphi},0\leq\rho<N\}$ be a ray segment,  $\mathcal{D}$ is some connected set
such that  $\Gamma\cap\mathcal{D}=\{0\}$ and
\begin{equation}
\label{2.19}
|\xi-t|\geq m\vert\rho e^{i\varphi_1}-|\xi|\vert \quad \text{ for }\xi\in\mathcal{D},~ t=\rho e^{i\varphi}\in \Gamma,
\end{equation} 
for some constants $m>0$ and $\varphi_1$ independent of $\epsilon$ .
Assume $g$ to
be a continuous function on $\Gamma$ with
$|g(t)-g(0)|\leq K \sqrt{|t|}$ for $t\in \Gamma\cap B_\nu$, $\nu$ is as in Remark 1.6.
Then
\begin{equation}\label{2.20}
\underset{\mathcal{D}\cap B_\nu}{\sup} \left\lvert\int_{\Gamma}\frac{g(t)-g(0)}{(
t-\xi)}dt-\int_{\Gamma}\frac{g(t)-g(0)}{
t}dt\right\rvert\leq C\left( (\sqrt{|\nu|})^{-1}\underset{\Gamma/B_\nu}{\sup}~ |g|+K\right)\sqrt{|\xi|}
\end{equation}
and
\begin{equation}\label{2.21}
\underset{\mathcal{D}\cap B_\nu}{\sup}\left\lvert\int_{\Gamma}\frac{g(t)-g(0)}{
t}dt\right\rvert\leq C\left( |\log \nu| \underset{\Gamma/B_\nu}{\sup} |g|+K\sqrt{|\nu|}\right),
\end{equation}
where constant $C$ depends on $\varphi_1$, and $ m$  only .
\end{lemma}
\begin{proof}

On $\Gamma$, $t=\rho e^{i\varphi},|dt|=d\rho$.
Breaking up the integral in (\ref{2.20}) into two parts:
\begin{equation}
\label{2.22}
\begin{split}
\int_{\Gamma}\frac{g(t)-g(0)}{(
t-\xi)}dt-\int_{\Gamma}\frac{g(t)-g(0)}{t}dt&=\xi\int_{\Gamma\cap B_\nu}\frac{g(t)-g(0)}{t(
t-\xi)}dt\\
&+\xi\int_{\Gamma/B_{\nu}}\frac{g(t)-g(0)}{t(
t-\xi)}dt.
\end{split}
\end{equation}
For the first integral in (\ref{2.22}), we use (\ref{2.19})
to obtain (on scaling $\rho$ by $|\xi|$):
\begin{equation*}
\left\vert\xi\int_{\Gamma\cap B_\nu}\frac{g(t)-g(0)}{t(
t-\xi)}dt\right\vert \leq \frac{K}{m}\sqrt{|\xi|}\int_0^1\frac{s^{-1/2}ds}{[(s\cos\varphi_1-1)^2+s^2\sin^2\varphi_1]^{1/2}}.
\end{equation*}
For the second integral we use (\ref{2.19})
to obtain
\begin{multline*}
\left\vert\xi\int_{\Gamma/ B_\nu}\frac{g(t)-g(0)}{t(
t-\xi)}dt\right\vert \leq|\xi|\frac{\underset{\Gamma/B_\nu}{\sup}~|g|}{m}\int_{|\nu|}^N\frac{ ds}{s[(s\cos\varphi_1-|\xi|)^2+s^2\sin^2\varphi_1]^{1/2}}\\
\leq C(\sqrt{|\nu|})^{-1}\left(\underset{\Gamma/B_\nu}{\sup}~ |g|\right)\sqrt{|\xi|}.
\end{multline*}
(\ref{2.21}) can be proved similarly.
\end{proof}

\begin{figure}
\begin{center}
\resizebox{0.75\textwidth}{!}
{\includegraphics{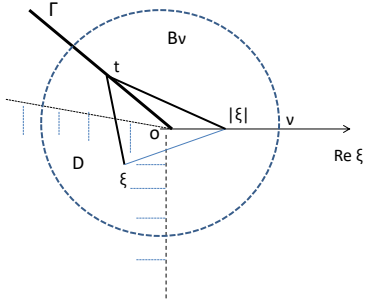}}
\end{center}
\caption{Relevant to (\ref{2.19}) in Lemma 2.15 and Remark 2.16 }
\label{fig:4}
\end{figure}
\begin{remark}
\label{rem:2.7}
Note if $\Gamma$ is in $\mathcal{D}_c'$, an angular subset of $\mathcal{D}_c$ (complement of $\mathcal{D}$), then (\ref{2.19})
hold (see Fig: \ref{fig:4}).
\end{remark}

\begin{definition}\label{def:2.6} Define
\begin{equation}
\label{2.23}
\begin{split}
\Omega_0=\biggl\{\xi:\xi  \text{ is below }
\{\xi=\rho e^{i(\pi-\varphi_0-\frac{1}{3}\mu)}\} \\
\cup \{ \xi=\rho e^{i(\varphi_0+\frac{1}{3}\mu)}\}\biggr\}.
\end{split}
\end{equation}
\end{definition}
\begin{remark}
\label{rem:2.8}
$\mathcal{R}\cup\{\text{ Im }\xi<0\}$ is an
angular subset of $\Omega_0$,
and $\Omega_0$ is itself an
angular subset of the region $\{\xi \text{ below }r_0\}$.
\end{remark}
\begin{lemma}
\label{lem:2.13}
Let $F \in \mathbf{A}_0^-$, then
$$\underset{\xi\in\Omega_0/B_\nu}{\sup}\abs{(\xi-2i)^{k+\tau}F_-^{(k)}(\xi)}
\leq K_2\frac{\underset{\xi\in\mathcal{R}/ B_\nu }{\sup}~|(\xi-2i)^{\tau}F|}{|\nu|^k}, ~k=0,1,2,$$
and
$$\underset{\xi\in\overline{\Omega_0\cap B_\nu }}{\sup}~\left\vert\frac{F'_-(\xi)-F'_-(0)}{\sqrt{\xi}}\right\vert \leq K_2 \left(\underset{\xi\in \overline{\mathcal{R}^- \cap B_\nu}}{\sup}~\left\vert\frac{F'(\xi)-F'(0)}{\sqrt{\xi}}\right\vert+|\nu|^{-3/2}\underset{\xi\in\mathcal{R}/ B_\nu }{\sup}~|F|\right),$$

where constant $K_2$ depends only on $\varphi_0$.
\end{lemma}
\begin{proof}
From Remarks \ref{rem:2.7}-\ref{rem:2.8}, conditions
(\ref{2.16})-(\ref{2.18})
hold with $\Gamma =r_0$ and $\mathcal{D}$ = $\Omega_0$.
Using $\norm{\bar F}_{0,r_0^-}\leq \norm{F}_0^-$, (\ref{2.18})
and applying Lemma \ref{lem:2.11}, with $g=\bar F'$, we obtain the proof.
\end{proof}
\begin{remark}
\label{rem:2.9}
From now on, we choose $F\in\mathbf{A}_{0,\delta}^-,$
with additional restriction,
\begin{equation}
\label{2.24}
\delta< \frac{K_3}{2}\equiv \frac{H_m}{2K K_2},
\end{equation}
where $H_m,K $ are as in (\ref{2.5}) with $\mathcal{D}=\mathcal{R}$,
while $K_2$ is defined as in
Lemma \ref{lem:2.13}. This ensures
\begin{equation}
\label{2.25}
|\xi-2i||F'_-+\bar H|\geq (H_m-K\norm{F'_-}_{1,\mathcal{D}})\geq (H_m-KK_2\norm{F}_0)
\geq C\frac{K_3}{2}>0;
\end{equation}
so $F'_-+\bar H\neq 0$
in $\mathcal{R}$. On the other hand, from Cauchy Integral Formula and Lemma 2.12 in \cite{Xie1}, we have $\norm{F'}_{1,\mathcal{D}}\leq C\norm{F}_0\leq \frac{K_1}{2}$. Therefore from Remark 2.8, $G(F,F_-)$ is analytic in $\mathcal{D}'$ which is subdomain of $\mathcal{R}$ and contains the real $\xi$ axis .

\end{remark}

\begin{remark}
\label{rem:2.10}
Using Lemmas \ref{lem:2.4} and
\ref{lem:2.13}, $G(F,F_-) (t) ~=~O(t^{-\tau})$
as $|t|\to\infty,t$ in any angular subset of
$\mathcal{R}\cap\Omega_0$ including the real axis.
\end{remark}

\begin{definition}\label{def:2.14}
If $F \in\mathbf{A}_{0,\delta}$,
define operator $I_1(F)$ so that for Im $\xi<0$,
\begin{equation}\label{2.26}
I_1(F) [\xi] =
-\frac{1}{2\pi }\int_{-\infty}^{\infty}\frac{G(F,F_-)(t)}{t-\xi}dt.
\end{equation}
\end{definition}

\begin{lemma}\label{lem:2.15}
If $F \in \mathbf{A}_{0,\delta}$ and $F$
is analytic in $\mathrm{Im} ~\xi ~>~0$  and also satisfies
(\ref{1.14}),
then in the region $\{~ \mathrm{Im} ~\xi<0~\}\cap\mathcal{R}$, $F$ satisfies
\begin{multline}
\label{2.27}
F'(\xi)~=~\frac{1}{2\epsilon^4}\big\{-\big[ G_2(F, I_1)^2(F_-'+\bar H)+2\epsilon^4 G_1(F'_-)\big] \\
+(F_-'
+\bar H)G_2(F,I_1)\sqrt{G_2(F,I_1)^2-4\epsilon^4}\big\}.
\end{multline}
\end{lemma}
\begin{proof}
Since conditions of Lemma \ref{lem:2.7} are satisfied,
(\ref{2.9}) holds.
Since $F$  is analytic in $\mathcal{R}\cup\{\text{Im }\xi>0\}$
with property (\ref{1.17}),
from Lemma \ref{lem:2.10}, $F_-={\bar F}$;
hence $I_1(F)(\xi)=I(F)(\xi)$. Therefore (\ref{2.9}) implies (\ref{2.27}).
\end{proof}
\begin{lemma}
\label{lem:2.16}
If $F \in \mathbf{A}_{0,\delta}$,
and
$F$ satisfies equation (\ref{2.27})
in $\mathcal{R}\cap
\{~ \mathrm{Im } ~\xi<0~\}$,
then $F$ is analytic in $\mathcal{R}\cup\{~\mathrm{ Im } ~\xi>0 ~\}$ and satisfies
(\ref{1.14})
on the real $\xi$ axis.
\end{lemma}
\begin{proof}
Note on using expression from (\ref{2.26}),
(\ref{2.27}) can be rewritten as:
\begin{equation}\label{2.28}
F(\xi)=-\epsilon^2 I_1(F)(\xi)
+i\epsilon^2 G(F,F_-)[\xi] \quad \text{for Im }~\xi<0,
\end{equation}
where operator $G$ is defined by (\ref{def:2.2}).
Analytically
continuing the above equation to the upper half plane, we have:
\begin{equation}
\label{2.29}
F(\xi)=
\frac{\epsilon^2}{2\pi }\int_{-\infty}^{\infty}\frac{G(F,F_-)(t)dt}{t-\xi}\quad
\text{~for Im }\xi>0;
\end{equation}
so $F(\xi)$ is analytic in the upper half plane.
From Lemma \ref{lem:2.10}, $F_- (\xi)={\bar F} (\xi)$; hence on
the real $\xi$-axis, $F_- (\xi) = F^*(\xi)$.
Equation (\ref{2.29}) reduces to:
\begin{equation*}
F(\xi)=-\frac{\epsilon^2}{\pi i}\int^{\infty}_{-\infty}
\frac{dt}{t-\xi}\frac{1}{\abs{F'(t)+H(t)}^{1/2}}\text{
Im }\left[F'(t)+H(t)\right] \qquad~\text{ for Im ~}\xi>0.
\end{equation*}
On taking the limit
Im ~$\xi ~\rightarrow~0^+$,
the above implies (\ref{1.14}).
\end{proof}
\noindent Because of Remark \ref{rem:1.5}
about equivalence of condition (iii)
to Im ~$F =0$ on $\{ \text{ Re ~}\xi =0 \} \cap \mathcal{R}$,
Lemmas \ref{lem:2.15} and \ref{lem:2.16} imply:
\begin{theorem}\label{thm:2.17}
The finger problem is equivalent to\\
\noindent {\bf Problem 1:} Find function $F
\in\mathbf{A}_{0,\delta}$, satisfying $\mathrm{Im} ~F ~=~0$ on
$\left \{~ \mathrm{Re} ~\xi =0 \right \} \cup \mathcal{R} $,
so that (\ref{2.27}) is satisfied
in $\mathcal{R}\cap\{~\mathrm{Im }~\xi <0\}$.
\end{theorem}

\subsection{Formulation of Problem 2}
\label{sec:2.2}
\noindent Let $\alpha_1>0$ be a fixed constant independent of $\epsilon$ so that $-\alpha_1 i\in \mathcal{R}$.
 
\noindent\begin{definition}\label{def:2.18}
Define two rays
\[
r^+_1=\{\xi:\xi=-\alpha_1 i+\rho e^{-i\varphi_0},
0<\rho<\infty\},
\]
\[
r^-_1=\{\xi:\xi=-\alpha_1 i+\rho e^{i(\pi +\varphi_0)},
0<\rho<\infty\}.
\]
$r_1=r_1^-\cup r_1^+$ is a directed contour from
left to right (see Fig.5).
\end{definition}

\begin{figure}
\begin{center}
\resizebox{0.75\textwidth}{!}
{\includegraphics{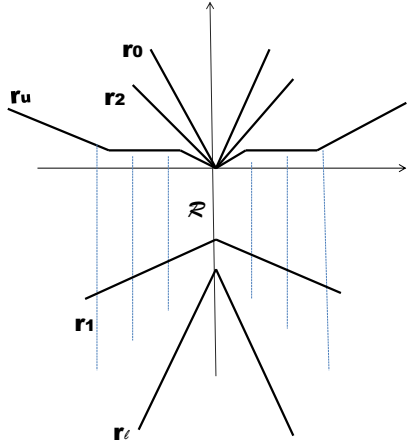}}
\end{center}
\caption{Rays defined in Subsection 2.2.}
\label{fig:5}
\end{figure}

\begin{definition}\label{def:2.19}
Let $F \in \mathbf{A}_{0,\delta}$.
For $\xi$ above $r_1$, define
\begin{equation}
\label{2.30}
F_+(\xi)= \frac{1}{2 \pi i} \int_{r_1} \frac{F(t) dt}{t-\xi}.
\end{equation}
\end{definition}
\begin{remark}
\label{rem:2.11}
$F_+$ as defined above
is analytic for $\xi$ above $r_1$.
Also, it is to be noted that
for
$F \in \mathbf{A}_0^-$ only, when
$F(\xi) = [F(-\xi^*)]^* $ is invoked to define $F$ on $\mathcal{R}^+$,
we can write
\begin{equation}
\label{2.31}
\begin{split}
F_+(\xi)&=
\frac{1}{2 \pi i} \left [ \int_{r_1^-}
\frac{F(t) dt}{t-\xi} + \int_{r_1^+} \frac{[F(-t^*)]^* dt}{t-\xi} \right ]
\\
&=~\frac{1}{2 \pi i} \int_{r_1^-}
\left [ \frac{F(t) dt}{t-\xi} - \frac{[F(t)]^* dt^*}{t^*+\xi} \right ].
\end{split}
\end{equation}
This expression is equivalent to (\ref{2.30}) when the symmetry condition:
Im ~$F ~=~0$ on $\{ \text{ Re ~}\xi =0 \} \cap \mathcal{R}$ is satisfied.
However, even without symmetry, (\ref{2.31})
still defines an analytic function for $\xi$ above $r_1$, with
possible singularity at $\xi = -i \alpha_1$.
It is easy to see that $F_+ (\xi)$ satisfies the symmetry condition on
$\{ \text{ Re ~}\xi = 0 \} \cup \mathcal{R} \cup \{ \text{ Im ~}\xi >-\alpha_1 \} $ even
when $F$ does not.
Further, if $F \in \mathbf{A}_0$ is also analytic in $\{ \text{ Im ~}\xi ~>~0 \} $,
then it is clear from (\ref{2.30}) on closing the contour from the top
that $F_+ (\xi) = F (\xi)$.

\end{remark}
\begin{definition}\label{def:2.20}
Define
\begin{multline}\label{2.32}
\Omega_1=\bigg\{\xi:\xi  \text{ is above }~
\{\xi=-\frac{\alpha_1}{2}i+\rho e^{i(\pi+\frac{\varphi_0}{2})}\}\cup\{\xi=-\frac{\alpha_1}{2}i+\rho e^{-i\frac{\varphi_0}{2}}\}\bigg\}.
\end{multline}
\end{definition}
\begin{remark}
\label{rem:2.12}
$\Omega_1$ is an angular subset of the region $\{\xi:
\xi\text{ above }r_1\}$, $r_1$ as in Definition \ref{def:2.18}.
\end{remark}
\begin{lemma}\label{lem:2.21}
Let $F\in \mathbf{A}_{0,\delta}^-,F'\in\mathbf{A}_{1,\delta_1}^-$, then
$$\underset{\xi\in\Omega_1}{\sup}
\left\lvert(\xi-2i)^{k+\tau}F_+^{(k)}(\xi)\right\rvert\leq
K_2  \norm{F}_0^- \quad \text{ for }k=0,1,2,$$
where $K_2>0$ is independent of $\epsilon$, $\lambda$.
\end{lemma}
\begin{proof}
The proof is very similar to that of Lemma \ref{lem:2.13}.
From Remarks \ref{rem:2.7}-\ref{rem:2.12}, conditions
(\ref{2.16})-(\ref{2.18})
hold with $\Gamma =r_1$ and $\mathcal{D}$ = $\Omega_1$.
Using $\norm{\bar F}_{0,r_0^-}\leq \norm{F}_0^-$, (\ref{2.31})
and applying Lemma \ref{lem:2.11}, with $g=F$, we obtain the proof.
\end{proof}

\begin{definition}\label{def:2.22}

We define two rays (see Figure \ref{fig:5}):
\[
r^+_2=\{\xi:\xi=\rho e^{i\varphi_0+\frac{1}{4}\mu },
0<\rho<\infty\},
\]
\[
r^-_2=\{\xi:\xi=\rho e^{i(\pi -\varphi_0-\frac{1}{4}\mu )},
0<\rho<\infty\}.
\]
$r_2$ is a directed contour from left to right on the path $r_2^-\cup
r_2^+$.
\end{definition}
\begin{remark}
\label{rem:2.13}
$r_2$ is an angular subset of $\Omega_0\cap\Omega_1$ and
$\mathcal{R}$ is below $r_2$.
\end{remark}
\begin{definition}
\label{def:2.12}
If $F\in\mathbf{A}_{0,\delta}^-$,
let $F'_-$ be given by (\ref{2.18}) and $F_+$ as in (\ref{2.30}).
Define operator $I_2$ so that
\begin{equation}
\label{2.33}
I_2(F)[\xi] =
-\frac{1}{2\pi }
\int_{r_2}\frac{G(F_+, F_-)(t)dt}{
t-\xi} \qquad
\text{for }\xi \text{ below }r_2 .
\end{equation}
\end{definition}
\begin{remark}
\label{rem:2.14}
Because of Lemmas \ref{lem:2.4}, \ref{lem:2.13} and \ref{lem:2.21},
$G(F_+, F_-)(t)\sim O(t^{-\tau})$ as $|t|\to \infty$
and analytic for
$t$ in any angular subset of $\Omega_0\cap\Omega_1$
that includes $r_2$; so $I_2(F)[\xi]$
is analytic below $r_2$. Also, from the
symmetry of each of $F_+$ and $F_-$, it is not difficult to
see that $I_2 (F)[\xi]$ also satisfies the symmetry condition on
$\{ \text{ Re ~}\xi = 0 \} \cap \mathcal{R}$.

\end{remark}

\noindent\begin{lemma}\label{lem:2.24}
Let $F\in\mathbf{A}_{0,\delta}^-$.
Then $I_2(F) \in\mathbf{A}$, and $
\norm{I_2(F)}_0\leq K_5,$
where $K_5>0$ is independent of $\epsilon$ and $\lambda$.
\end{lemma}
\begin{proof}
From Lemma \ref{lem:2.13} and Lemma \ref{lem:2.21},
\[
\norm{F_-^{(k)}}_{k,r_2}\leq
K_2\norm{F}_0^-~~,~~ \norm{F_+^{(k)}}_{k,r_2}\leq
K_2 \norm{F}_0^-.
\]
Applying Lemma \ref{lem:2.4} (with
$\mathcal{D}=r_2$) and Lemma \ref{lem:2.11} (with
$\Gamma = r_2$ and $\mathcal{D}=\mathcal{R}$)
to (\ref{2.33}),
we complete the proof.
\end{proof}

\begin{lemma}\label{lem:2.26}
If $F \in \mathbf{A}_{0,\delta},$ and
$F$ satisfies (\ref{2.27})
in $\mathcal{R}\cap\{~\mathrm{Im }~\xi<0\}$, then for $\xi\in\mathcal{R}$,
$F$ satisfies
\begin{multline}
\label{2.34}
F'(\xi)~=~\frac{1}{2\epsilon^4}\big\{-\big[ G_2(F, I_2)^2(F_-'+\bar H)+2\epsilon^4 G_1(F'_-)\big] \\
+(F'_-+\bar H)G_2(F,I_2)\sqrt{G_2(F,I_2)^2-4\epsilon^4}\big\}.
\end{multline}
\end{lemma}
\begin{proof}
If $\delta $ is small enough,
from Lemmas \ref{lem:2.13} and \ref{lem:2.21},
$|(\xi-2i)^{1+\tau}F_-'|$ and
$|(\xi-2i)^{1+\tau} F^\prime_+ |$ 
are each small in the domain
$\Omega_0 \cap \Omega_1$
which contains the region between $r_2$ and $r_1$; hence
$F_-'+\bar H\neq 0$ and $F_+^\prime + H ~\neq 0$ in that domain.
From Lemma \ref{lem:2.16}, $F$ is analytic in $\{ \text{ Im ~}\xi ~>~0 \} $; hence $F_+ = F$.
By deforming the contour $r_2$ in
(\ref{2.33})
back to the real axis, it follows $I_2(F)(\xi)=I_1(F)(\xi)$, for Im $\xi<0$.  By analytic continuation, $F$ satisfies
(\ref{2.34}) for $\xi\in\mathcal{R}$.
\end{proof}

\noindent {\bf Problem 2}: Find function
$F\in\mathbf{A}_{0,\delta}^-$
so that $F$ satisfies the symmetry
condition Im ~$F =0 $ on $\mathcal{R} \cap \left \{ ~\text{Re ~}\xi = 0 \right \} $
and equation (\ref{2.34})
in $\mathcal{R}^-$.

\begin{theorem}\label{thm:2.41}
Let $F\in\mathbf{A}_{0,\delta}^-$.
If $F(\xi)$
satisfies the symmetry condition
$\mathrm{Im} ~F =0 $ on $\mathcal{R} \cap \left \{~ \mathrm{Re} ~\xi = 0 \right \} $
and the
equation (\ref{2.34}) in $\mathcal{R}^-$,
then for sufficiently small $\epsilon$ and $\delta$ ,
$F$ is a solution to {\bf Problem 1} (and hence a solution to
the original Finger Problem).
\end{theorem}
The proof of Theorem \ref{thm:2.41} will be given
later after  several lemmas.

\begin{lemma}\label{lem:2.42}
Assume $F\in\mathbf{A}_{0,\delta}^-$
and  $F$ satisfies the
integral equation (\ref{2.34})
in $\mathcal{R}^-$ as well as the symmetry
condition $ \mathrm{Im} ~F (\xi)= 0$ for $\xi \in \{ ~\mathrm{Re} ~\xi =0\} \cap \mathcal{R}$.
Let $U(\xi)=F(\xi)-F_+ (\xi)$, then\\
(1) $U(\xi)$ is analytic in $\mathcal{R}\cap\{~ \mathrm{Im}~ \xi>0\}$.\\
(2) For $\xi\in\mathcal{R}^-\cap \{~ \mathrm{Im} ~\xi ~>~0 \}$, $U$ satisfies:
\begin{equation}
\label{2.35}
\epsilon^2 U'(\xi)+ \tilde Q(\xi) U(\xi)=\bar M(U)+M(U),
\end{equation}
where
\begin{equation}\label{2.36}
\tilde Q(\xi)=\frac{2i
H^{3/2}\bar H^{1/2}}{H+\bar H}=i\frac{(\xi+i\gamma)^{3/2}(\xi-i\gamma)^{1/2}}{\xi (\xi^2+1)},
\end{equation}

\begin{equation}
\label{2.37}
\begin{split}
&\bar M(\xi)=\bar M (U) [\xi]
\equiv [-iR_1(\xi)+\tilde Q(\xi)]U , \\
&M(U)(\xi):
=-\frac{\epsilon^2i}{2\pi}\mathcal{R}_1(\xi)\int_{-\infty}^\infty
\frac{\mathcal{R}_1^{-1}(t) U'(t)}{t-\xi}dt \qquad
\text{for $\mathrm{Im }~\xi ~>~0 $},
\end{split}
\end{equation}
$\mathcal{R}_1 (\xi) = R_1 (F, F_-, F_+) [\xi]$, where operator $R_1$
is defined by:
\begin{multline}
\label{2.38}
R_1(F,F_-,F_+)\\
=\frac{[(F'_++H)^{1/2}+(F'+H)^{1/2}](F_+'+H)^{1/2}(F'+H)^{1/2}(F_-'+\bar H)^{1/2}}{[(F'_-
+\bar H)+(F'+H)^{1/2}(F_+'+H)^{1/2}]}.
\end{multline}
\end{lemma}
\begin{proof}
Since
each of $F$ and $F_+$ are analytic in $\mathcal{R}\cap\{~ \text{ Im ~}\xi>0\}$, it follows that
$U = F-F_+$ is also analytic in $\mathcal{R}\cap\{ \text{ Im ~}\xi>0\}$; hence statement (1)
follows.
Since $F$ satisfies (\ref{2.34}) in $\mathcal{R}^-$, the symmetry condition and the Schwarz reflection principle
that relates $F$ and its derivatives in $\mathcal{R}^+$ to
their values in $\mathcal{R}^-$ guarantees that
(\ref{2.34})
is satisfied in $\mathcal{R}$. But equation (\ref{2.34}) can be rewritten as:
\begin{equation}
\label{2.39}
F(\xi)=-\epsilon^2 I_2(\xi)
+i\epsilon^2 G(F, F_-)(\xi).
\end{equation}
Then, on deforming the contour
for $I_2$ from $r_2$ to $(-\infty,\infty)
$,
\begin{equation}
\label{2.40}
F(\xi) = -\epsilon^2 I_3 (\xi)
-i \epsilon^2 \left [ G(F_+, F_-)(\xi) - G (F, F_-)(\xi) \right ]
~\text{ for $\xi$ above $(-\infty,\infty)$},
\end{equation}
where
\begin{equation}
\label{2.41}
I_3(\xi)=-
\frac{1}{2\pi}\int_{-\infty}^\infty
\frac{G(F_+, F_-)(t)}{t-\xi} dt \qquad
\text{for $\xi$ above $r_1$ }.
\end{equation}
It is clear that $I_3 (\xi)$ is analytic above $r_1$; indeed from
contour deformation of (\ref{2.30}) and
(\ref{2.41}) and analyticity and decay properties of
$G(F_+, F_-)$ on $\mathcal{R}$ itself,
it is clear that $I_3 \in \mathbf{A}_0$ and analytic in
$\mathcal{R} \cup \{~\text{ Im ~}\xi ~>~0 \}$.
Substituting for $F$ from (\ref{2.41}) into
(\ref{2.34}), it follows that for Im ~$\xi~>~0$,
\begin{equation}
\label{2.42}
\begin{split}
&F_+ (\xi) = -\epsilon^2 I_3 (\xi) - \frac{\epsilon^2}{2 \pi}
\int_{-\infty}^\infty \frac{[G(F_+, F_-)[t] - G(F, F_-)[t]]}{t-\xi} dt \\
& =\frac{\epsilon^2}{2\pi}\int_{-\infty}^\infty
\frac{G(F, F_-)(t)}{t-\xi} dt
~~\qquad \text{for $ \text{ Im ~}\xi >~0$}.
\end{split}
\end{equation}
Subtracting (\ref{2.42}) from (\ref{2.40}),
we obtain for Im ~$\xi ~>~0$:
\begin{equation}
\begin{split}
\label{2.43}
U (\xi) ~=~&-i \epsilon^2
\left [G(F_+, F_-)(\xi)- G(F, F_-) (\xi)
\right ] \\
&~+ \frac{\epsilon^2}{2\pi} \int_{-\infty}^\infty
\frac{G(F_+, F_-)(t)- G(F, F_-) (t)}{t-\xi} dt.
\end{split}
\end{equation}

Using (\ref{2.3}), and the definition of $G_1$ and $G_4$ in (\ref{2.41}) and
(\ref{2.38}),
\begin{equation}
\label{2.44}
G(F, F_-)-G(F_+, F_-)=R_1^{-1} (F, F_-)U'.
\end{equation}
From (\ref{2.44}), (\ref{2.43}),
$U(\xi)$ satisfies the equation (2.35) for
Im ~$\xi ~>~0$.

\end{proof}

\begin{definition}
We define two rays $r_3^{-}$ and $r_3^{+}$:
$$ r_3^- = \{\xi: \xi=r e^{i\left(\pi-\frac{\phi_0}{2}\right)},r>0\}, $$
$$ r_3^{+} = \{\xi: \xi=r e^{\frac{i\phi_0}{2}},r>0\}; $$
and $\Omega_3$ is defined to be the region between $r_3=r_3^-\cup r_3^+$ and $r_u$.
\end{definition}

\begin{definition}\label{L2.4}
We define
\begin{gather}\label{2.45}
g_1(\xi)=\exp \left(-
\frac{P(\xi)}{\epsilon^2}\right),\\
g_2(\xi)=\exp \left(
\frac{P(\xi)}{\epsilon^2} \right),
\end{gather}
\end{definition}
where
\begin{equation}\label{2.47}
P(\xi)~=~\int^\xi \tilde Q(t)dt
~=~\int i\frac{(\xi+i\gamma)^{3/2}(\xi-i\gamma)^{1/2}}{\xi (\xi^2+1)}d\xi.
\end{equation}

\begin{lemma}
Let $U$ be as in Lemma 2.39  , then $U$ satisfies
\begin{equation}\label{2.48}
U=\frac{1}{\epsilon^2}g_1(\xi)\int_{\infty}^\xi
(\bar M(t)+M(t))g_2(t)dt.
\end{equation}
\end{lemma}
\begin{proof}
Since $U$ satisfies (\ref{2.35}), we have
$$U(\xi)=Cg_1(\xi)+\frac{1}{\epsilon^2}g_1(\xi)\int_{\infty}^\xi
(\bar M(t)+M(t))g_2(t)dt.$$
$C$ has to be zero since $U(\infty)=0$ and $g_1(\infty)\neq 0$ from Lemma 6.1.
\end{proof}

\begin{remark}
\label{2.17.9}
We will show that $U = 0$. Since $F$, $F^+$ and hence $U$ is known
to be analytic in $\mathcal{R}$ and continuous upto its boundary,
it is enough to show that $U = 0$ on $\Omega_3^-$.
We will do so by showing that equation (\ref{2.48})  forms a contraction map in the space of
functions $U$ on $\Omega_3^-$, with norm
$$ \| U \|_{\Omega_3} = \sup_{\xi \in \Omega_3} ~|\xi-2i|^\tau |U (\xi)|+\epsilon^2 \underset{\overline{B_\nu\cap \Omega_3}}{\sup}~\left\vert\frac{U'(t)-U'(0)}{\sqrt{t}}\right\vert
 +|U'(0)|$$
It is to be noted that the integration in (\ref{2.48})
can be performed on the path $\mathcal{P}(\xi,\infty)$ contained in $\Omega_3^-$, so that on
this path on any $\mathbf{C}^1$ segment,
$\frac{d}{ds} ~\text{Re} ~P(t(s)) ~<~0$ for arc length $s$ increasing
in the direction of $\infty$, Lemmas 2.50-2.53
are valid and hence the integral operator in (\ref{2.48})  is bounded when
restricted to functions on $\Omega_3^-$.  Further note that since
$U$ are analytic in $\mathcal{R} \cup \{ ~\text{Im ~}\xi ~<~0 \}$ and
satisfies the symmetry condition,
it follows that $U$ on $\Omega_3^{-}$ completely determines
$U$ and its derivatives for $\xi$ real. This is crucial in
controlling $\mathcal{M} $ on $r_3^-$, as necessary.
\end{remark}

\begin{lemma}\label{lem:2.43}
Let $F\in\mathbf{A}_{0,\delta}$,
then for $\xi\in\mathcal{R}$
\begin{equation}\label{2.49}
 |R_1 (F, F_-)(\xi)|\leq C /|\xi-2i|, ~|1 / R_1 (F, F_-)(\xi)|\leq C |\xi-2i|,
\end{equation}
where $C$ is independent of $\epsilon$ and $\lambda$.
\end{lemma}
\begin{proof}
The lemma follows from (2.38) and Lemma 2.39.
\end{proof}

\begin{lemma}\label{lem:2.45}
Let $U(\xi)$ be as in Lemma \ref{lem:2.42}, then
\begin{equation}\label{2.50}
\underset{\xi\in \mathcal{D}}{\sup} ~|\xi-2i|^{1+\tau}|U'|(\xi)|
\leq \frac{C}{|\nu|}\norm{U}_{0, r_3^-} ~,~~
\end{equation}
where $\mathcal{D} =(-\infty,-\nu) $, $C$ is independent of $\epsilon$ and $\lambda$.
\end{lemma}
\begin{proof}
Since $U$ is analytic in the region under $r_3$ and continuous upto
its boundary,
by the Cauchy Integral Formula:
\begin{equation*}
U'(\xi)=-\frac{1}{2\pi i}\int_{r_3} \frac{U(t)}{(t-\xi)^{2}}dt \qquad
\text{for $\xi$ in $\mathcal{D}$ }.
\end{equation*}
Applying Lemma 2.11 in \cite{Xie1}, with $\Gamma$ chosen to be $r_u$, we have $ |\xi-2i|^{1+\tau}|U'|(\xi)|
\leq C \norm{U}_{0, r_3^-}$ for $\xi\in (-\infty,-3)$ .\\
For $\xi\in (-\infty,-3)$, we split the integral into
\begin{equation*}
\begin{split}
U'(\xi)&=-\frac{1}{2\pi i}\left(\int_{r_3\cap B_3}+ \int_{r_3/B_3}\right)\frac{U(t)}{(t-\xi)^{2}}dt\\
&=:U_1+U_2.
\end{split}
\end{equation*}
Applying Lemma 2.11 in \cite{Xie1}, with $\Gamma$ chosen to be $r_u/B_3$, we have $ |U_2|(\xi)|
\leq C \norm{U}_{0, r_3^-}$ for $\xi\in (-3,-\nu)$ .\\
For  $t=re^{\theta}\in r_3\cap B_3$ and $\xi\in (-3,-\nu)$, using $|t-\xi|^2=(r\cos\theta-\xi)^2+r^2\sin^2\theta$, we have
\begin{equation*}
\begin{split}
|U_1(\xi)|&=\left\vert-\frac{1}{2\pi i}\int_{r_3\cap B_3}\frac{U(t)}{(t-\xi)^{2}}dt\right\vert\\
&\leq C\left\vert\int_0^3\frac{\underset{r_3}{\sup}|U(t)|}{(r\cos\theta-\xi)^2+r^2\sin^2\theta}dr\right\vert\leq C\frac{\underset{r_3}{\sup}~|U(t)|}{|\xi|};
\end{split}
\end{equation*}
hence the lemma follows.

\end{proof}
\begin{lemma}
Let $U(\xi)$ be as in Lemma \ref{lem:2.42}, then
\begin{equation}\label{2.51}
\underset{\xi\in (-\nu,0)}{\sup} |U'(\xi)-U'(0)|
\leq \frac{C\sqrt{|\xi|}}{|\nu|^{3/2}}~\norm{U}_{\Omega_3} ,
\end{equation}
$C>0$ is independent of $\epsilon$ and $\lambda$.
\end{lemma}
\begin{proof}
\begin{equation*}
U'(\xi)=-\frac{1}{2\pi i}\int_{r_4} \frac{U'(t)}{(t-\xi)}dt
\text{ for $\xi$ in $(-\nu,0)$ },
\end{equation*}
where $r_4$ is a line between $r_3$ and $r_u$ defined by\\
$$ r_4^- = \{\xi: \xi=r e^{i(\pi-2\phi_0/3)},r>0\}, $$
$$ r_4^{+} = \{\xi: \xi=r e^{i(2\phi_0/3)},r>0\}. $$
Applying Lemma 2.15, we obtain
\begin{equation*}
|U'(\xi)-U'(0)|\leq C\sqrt{|\xi|}\left( \frac{\underset{r_4/B_\nu}{\sup}|U'(t)|}{\sqrt{|\nu|}}+\norm{U}_{0, r_4}\right).
\end{equation*}
From Cauchy integral formula, for $t\in r_4/ B_\nu$  we have
\begin{equation*}
U'(t)=-\frac{1}{2\pi i}\int_{\partial\Omega_3} \frac{U(s)}{(s-t)^2}ds;
\end{equation*}

since $|t-s|\geq C\vert|s|e^{i\theta}-|t|\vert$, for $t\in r_4/B_\nu$ and $s\in \partial\Omega_3$ we have
\begin{equation*}
\underset{r_4/B_\nu}{\sup}|U'(t)|\leq C\frac{\underset{\partial\Omega}{\sup}~|U(s)|}{|\nu|};
\end{equation*}
hence the lemma follows.
\end{proof}
\begin{lemma}
Let $U(\xi)$ be as in Lemma \ref{lem:2.42}, then $\underset{(-\nu,0)}{\sup}~|U'(\xi) |\leq \frac{C}{|\nu|}~\norm{U}_{\Omega_3} $,
where $C>0$ is a constant independent of $\epsilon$ and $\lambda$.
\end{lemma}
\begin{proof}
The lemma follows from the Cauchy Integral Formula and the above lemma.
\end{proof}

\begin{lemma}\label{lem:2.46}
Let $F\in\mathbf{A}_{0,\delta}$.
Let ${\bar M}(U)$ and $M(U)$ be as
defined in (\ref{2.37}), then
\begin{equation}
\label{2.52}
 |{\bar M}(U)(\xi)|\leq
 C\delta_2 (1+|\xi|)^{-1-2\tau} \norm{U}_{\Omega_3},
\end{equation}
\begin{equation}
\label{2.53}
 |M(U)(\xi)|\leq
 C\frac{\epsilon^2}{|\nu|^2} (1+|\xi|)^{-1-\tau} \norm{U}_{\Omega_3} \qquad \text{ for }\xi\in\Omega/B_\nu,
\end{equation}

\begin{equation}
\label{2.54}
 \left\vert \frac{d}{d\xi}M(U)(\xi)\right\vert\leq
 C\frac{\epsilon^2}{|\nu|^2} (1+|\xi|)^{-2-\tau} \norm{U}_{\Omega_3} \qquad \text{ for }\xi\in\Omega/B_\nu,
\end{equation}
and 

\begin{equation}\label{2.55}
\underset{\xi\in (-\nu,0)}{\sup}~ |M'(\xi)-M'(0)|
\leq \frac{C\sqrt{|\xi|}}{|\nu|^{2}}\norm{U}_{\Omega_3}\qquad \text{ for }\xi\in\Omega_3\cap B_\nu,
\end{equation}
where $C$ is some constant independent of $\epsilon$ and
\begin{equation}
\label{2.56}
\delta_2=\bigg\{
\left(\norm{F}_0+1\right)^{1/2}-1\bigg\}.
\end{equation}
\end{lemma}
\begin{proof}
The (\ref{2.52}) follows from
\begin{equation}\label{2.57}
|(-iR_1+\tilde Q)|\leq C|t-2i|^{-1-\tau}\delta_2~\qquad
\text{for $t \in \Omega_3$}.
\end{equation}

Using Lemma \ref{lem:2.43} and Lemma \ref{lem:2.45}:
\begin{equation}\label{2.58}
|R_1^{-1}U'(t)|\leq \frac{C}{|\nu|}|t-2i|^{-\tau}\norm{U}_0~ \qquad 
\text{for $t \in (-\infty, -\nu)$} .
\end{equation}
Applying Lemmas \ref{lem:2.11},
and Lemmas \ref{lem:2.43} - \ref{lem:2.45} to (\ref{2.37}), we obtain (\ref{2.53}) and (\ref{2.54}). (\ref{2.55}) follows  from Lemmas 2.45 - 2.47.
\end{proof}

\begin{definition}\label{L2.7}
Let $\mathcal{Q}$ be any connected  set in the complex $\xi$-plane, we introduce norms:
$\norm{F(\xi)}_{j,\mathcal{Q}}
:=\underset{\xi\in\mathcal{Q}}{\sup}~\abs{(\xi-2i)^{j+\tau}F(\xi)},j=0,1,2$;
and $\mathbf{C}(\mathcal{Q})$ is the function space of all continuous functions on $\mathcal{Q}$.
\end{definition}

In order to prove Theorem 2.38 using the strategy described in Remark 2.43, we need to control the integral operator in equation (\ref{2.48}). Due to the properties of $P(\xi)$ (see the appendix), we need to break the integral from $0$ to $\infty$ into several integrals along special paths, each of these integrals then can be estimated accordingly (see Lemmas 2.50 - 2.54 below).

\begin{lemma}
If $N \in\mathbf{C}(\mathcal{Q}), N' \in\mathbf{C}(\mathcal{Q})$ where $\mathcal{Q}$ is a path $\mathcal{P}(\xi, \infty)$ on which $\mathrm{Re} ~(P(t)-P(\xi))$ decreases monotonically  as $\mathrm{Re}~ t$ increases. Let
\begin{equation}
f_1(\xi):=\frac{1}{\epsilon^2}g_1(\xi)\int^{\xi}_{\infty}
N(t)g_2(t)dt.
\end{equation}
Then $f_1(\xi)\in\mathbf{C}(\mathcal{Q})$ and 
\begin{equation}\label{2.60}
\norm{f_1}_0
\leq \frac{C}{\epsilon^2}~\norm{N}_1,
\end{equation} 
and $\norm{f_1}_0
\leq C~(\norm{N}_1+\norm{N'}_2),$
where $C$ is independent of $\epsilon$ and $\lambda$.
\end{lemma}
\begin{proof}
\begin{multline*}
|f_1(\xi)|=\left\lvert\frac{1}{\epsilon^2}
\int_{\mathcal{P}(\xi, \infty)} N(t)
\exp\{-\frac{1}{\epsilon^2}(P(\xi)-P(t))\}dt
\right\rvert\\
\leq \frac{1}{\epsilon^2}\norm{N}_1
\int^{\infty}_{|\xi|}|(t-2i)^{-1-\tau}|dt
\leq \frac{C}{\epsilon^2}\norm{N}_1 |\xi-2i|^{-\tau}.
\end{multline*}
By integration by parts, we have
\begin{multline*}
|f_1(\xi)|=\left\lvert\frac{1}{\epsilon^2}
\int_{\mathcal{P}(\xi, \infty)} N(t)
\exp\{-\frac{1}{\epsilon^2}(P(\xi)-P(t))\}dt
\right\rvert\\
=\left\lvert\frac{1}{\epsilon^2}
\int_{\mathcal{P}(\xi, \infty)} \frac{N(t)}{-
\frac{d}{ds}  P(t)}
d\left(\exp\{-\frac{1}{\epsilon^2}(P(\xi)-P(t))\}\right)
\right\rvert\\
\leq \norm{N}_1
\times\int^{\infty}_{0}\frac{|(t-2i)^{-1-\tau}||
\frac{d^2}{dt^2} P(t)|}{|
\frac{d}{dt} P(t)|^2}
\left[\exp~\{-\frac{1}{\epsilon^2}\text{ Re } (P(\xi)- P(t))\}\right]dt\\
+\norm{N'}_2
\times\int^{\infty}_{0}\frac{|(t-2i)^{-2-\tau}|}{
|\frac{d}{dt} P(t)|}
\left[\exp~\{-\frac{1}{\epsilon^2} \text{ Re }(P(\xi)- P(t))\}\right]dt.
\end{multline*}
Since
$$\left\lvert\frac{d}{dt} ~~P(t)\right\rvert\geq C |t(s)-2i|^{-1},$$
$$\left\lvert\frac{d^2}{dt^2} ~~P(t)\right\rvert\leq C |t(s)-2i|^{-2},$$
$$
\frac{|(t-2i)^{-2-\tau}|}{\frac{d}{dt}
 P(t)} \leq C |\xi-2i|^{1-\tau},$$
we obtain
$\norm{f_1}_0 \leq C (\norm{N}_1+\norm{N'}_2)$ and the lemma follows.
\end{proof}

If the estimate of the derivative of $N$ is not available and   due to the factor $\frac{1}{\epsilon^2}$ in (\ref{2.60}), Lemma 2.50 does not provide the estimate needed for the integral operator to be a contraction. By breaking up the integral and using properties of $P(\xi)$ ( Lemma 6.1, Lemma 6.3 and Lemma 6.6 in the appendix), we are able to obtain better estimates for the integral operator in the following two lemmas.

\begin{lemma}
Let $R_\epsilon >R$ be a number which will be chosen later and be dependent on $\epsilon$. If $N(t)\leq C |t|^{-\sigma}$ for $t\in \mathcal{P}(\xi,R^l_\epsilon)$, where $R^l_\epsilon=r_l^-\cap \{|\xi|=R_\epsilon\}$ and $\mathcal{P}(\xi,R^l_\epsilon )\subset \mathcal{R}^-\cap \{R\leq ~|\xi|\leq R_\epsilon\}$ is a path on which $\mathrm{Re}~ (P(t)-P(\xi))$ decreases monotonically  as $\mathrm{Re} ~t$ increases, then
\begin{equation*}
\begin{split}
&f_2(\xi):=\frac{1}{\epsilon^2}g_1(\xi)\int^{\xi}_{R^l_\epsilon}
N(t)g_2(t)dt\in\mathbf{C}\left(\mathcal{P}(\xi,R^l_\epsilon)\right),\\
&\text{ and }\norm{f_2}_0
\leq C~R_\epsilon^{(2-\sigma+\tau)}\underset{\mathcal{P}(\xi,R^l_\epsilon)}{\sup}(|t|^\sigma |N|),
\end{split}
\end{equation*}
\end{lemma}
$C$ is independent of $\epsilon$ and $\lambda$.
\begin{proof}

\begin{multline*}\label{2.59}
|f_2(\xi)|=\left\lvert\frac{1}{\epsilon^2}
\int_{\mathcal{P}(\xi, R^l_\epsilon)} N(t)
\exp~\{-\frac{1}{\epsilon^2}(P(\xi)-P(t))\}dt
\right\rvert\\
\leq \sup{(|t|^\sigma |N|)}
\times\int^{1}_{0}\frac{|(t(s)|^{-\sigma}}{-
\frac{d}{ds} \text{ Re }P(t(s))}
d\left[\exp~\{-\frac{1}{\epsilon^2}\text{Re } (P(\xi)-P(t(s)))\}\right].
\end{multline*}
Since
$$-\frac{d}{ds} ~\text{Re} ~P(t(s))= - \text{ Re }\left( P'(t)t'(s)\right)\geq C |t(s)|^{-2},$$
$$
\frac{|t(s)^{-\sigma}|}{-\frac{d}{ds}
\text{ Re }P(t(s))} \leq C |t|^{2-\sigma},$$
we have
$\norm{f_2}_0 \leq C R_\epsilon^{2-\sigma+\tau} \underset{\mathcal{P}(\xi,R^l_\epsilon)}{\sup}{|t^\sigma N(t)|}$ .
\end{proof}

\begin{lemma}\label{L2.9}
Let $N \in\mathbf{C}(\mathcal{P}(\xi,R^l))$, where $R^l=r^-_l\cap \{\xi: |\xi|=R\}$ and $\mathcal{P}(\xi,R^l)\subset \mathcal{R}^-\cap \{ |\xi|\leq R\}$ is a path on which $\mathrm{ Re} ~(P(t)-P(\xi))$ decreases monotonically  as $\mathrm{Re} ~t$ increases, then
\begin{equation*}
f_3(\xi):=\frac{1}{\epsilon^2}g_1(\xi)\int^{\xi}_{R^l}
N(t)g_2(t)dt\in\mathbf{C}(\mathcal{P}(\xi,R^l)),\text{ and }\norm{f_3}_0
\leq C~\norm{N}_1,
\end{equation*}
where $C$ is a constant independent of $\epsilon$.
\end{lemma}
\begin{proof}

\begin{multline*}
|f_3(\xi)|=\left\lvert\frac{1}{\epsilon^2}
\int_{\mathcal{P}(\xi, R^l)} N(t)
\exp~\{-\frac{1}{\epsilon^2}(P(\xi)-P(t))\}dt
\right\rvert\\
\leq \underset{\mathcal{P}(\xi,R^l)}{\sup}{N}
\times\int^{1}_{0}\frac{1}{-
\frac{d}{ds} \text{ Re }P(t(s))}
d\left[\exp\{-\frac{1}{\epsilon^2}\text{ Re }(P(\xi)- P(t(s)))\}\right]
\end{multline*}
Since for $-R\leq \text{ Re~ }t(s)\leq 0$, we have
$$-\frac{d}{ds} ~\text{Re ~}P(t(s))= -\text{Re }\left( P'(t)t'(s)\right)\geq C ,$$
and
$$
\frac{1}{-\frac{d}{ds}
\text{ Re }P(t(s))} \leq C .$$
So
$\norm{f_3}_0 \leq K_1  \underset{\mathcal{P}(\xi,R^l)}{\sup}{ |N(t)|}$ .
\end{proof}

In order to estimate the integral operator in a neighborhood of the origin, we use the special property $P(\xi)\sim -\gamma^2\log~\xi$ for $\xi$ near the origin in the following lemma.

\begin{lemma}
Let $N$ be a continuous function on $\Omega^-_3$ and $\norm{N}_{\Omega_3}$ be defined as in  Remark 2.44. Let
\begin{equation*}
f_1(\xi):=\frac{1}{\epsilon^2}g_1(\xi)\int^{\xi}_{\infty}
N(t)g_2(t)dt,
\end{equation*}
then
\begin{equation*}
\epsilon^2 \underset{\Omega^-_3\cap B_\nu}{\sup}\left\vert \frac{f'_1(\xi)-f'_1(0)}{\sqrt{\xi}}\right\vert\leq C \norm{N}_{\Omega_3},
\end{equation*}
where $C$ is a constant independent of $\epsilon$.
\end{lemma}
\begin{proof}
\begin{equation*}
f'_1(\xi)=\frac{1}{\epsilon^2}N(\xi)-\frac{P'(\xi)}{\epsilon^4}g_1(\xi)\int^{\xi}_{\infty}
N(t)g_2(t)dt,
\end{equation*}
\begin{equation*}
f'_1(\xi)-f_1'(0)=\frac{1}{\epsilon^2}[N(\xi)-N(0)]-\frac{P'(\xi)}{\epsilon^4}g_1(\xi)\int^{\xi}_{\infty}
[N(t)-N(0)]g_2(t)dt,
\end{equation*}
we can break up the integral on the right hand side into
\begin{equation*}
\begin{split}
&\frac{P'(\xi)}{\epsilon^4}g_1(\xi)\int^{\xi}_{\infty}[N(t)-N(0)]g_2(t)dt\\
&=\frac{P'(\xi)}{\epsilon^4}g_1(\xi)\int^{\nu_1}_{\infty}[N(t)-N(0)]g_2(t)dt+\frac{P'(\xi)}{\epsilon^4}g_1(\xi)\int^{\xi}_{-\nu_1}[N(t)-N(0)]g_2(t)dt\\
&=\frac{P'(\xi)}{\epsilon^4}g_1(\xi)g_2(-\nu_1)\left(g_1(-\nu_1)\int^{-\nu_1}_{\infty}[N(t)-N(0)]g_2(t)dt\right)\\
&+\frac{P'(\xi)}{\epsilon^4}g_1(\xi)\int^{\xi}_{-\nu_1}[N(t)-N(0)]g_2(t)dt\\
&:=w_1(\xi)+w_2(\xi).
\end{split}
\end{equation*}
Using Lemma 2.50, we have
$$\left\vert g_1(-\nu_1)\int^{-\nu_1}_{\infty}[N(t)-N(0)]g_2(t)dt\right\vert\leq C\epsilon^2 \norm{N}_{\Omega_3},$$
and using the fact that $P(t)\sim -\gamma^2\log t$ for $|t|\leq \nu_1$, we obtain for $\xi\in B_\nu, \nu=O(\epsilon^{4/3})$,
\begin{equation*}
\left\vert\frac{P'(\xi)}{\epsilon^4}g_1(\xi)g_2(-\nu_1)\right\vert\leq C|\xi|^{\frac{\gamma^2}{\epsilon^2}-1}|\nu_1|^{\frac{\gamma^2}{\epsilon^2}}\leq C\sqrt{|\xi|}\epsilon^{\frac{\gamma^2}{2\epsilon^2}-2};
\end{equation*}
hence $|w_1(\xi)|\leq C\sqrt{|\xi|}\norm{N}_{\Omega_3}\epsilon^{\frac{\gamma^2}{2\epsilon^2}}$.\\
Using the fact that $P(t)\sim -\gamma^2\log t$ and $g_1(\xi)\sim \xi^{\frac{\gamma^2}{\epsilon^2}}$ for $|\xi|\leq \nu_1$ and $|N(t)-N(0)|\leq \norm{N}_{\Omega_3} |t|^{1/2}$  we obtain
\begin{equation*}
\begin{split}
|w_2(\xi)|&=\left\vert\frac{P'(\xi)}{\epsilon^4}g_1(\xi)\int^{\xi}_{-\nu_1}N(t)g_2(t)dt\right\vert\\
&=\left\vert\frac{C}{\epsilon^4}\xi^{\frac{\gamma^2}{\epsilon^2}-1}\int^{\xi}_{-\nu_1}N(t)t^{\frac{\gamma^2}{\epsilon^2}}dt\right\vert \\
&\leq \frac{C}{\epsilon^2}\sqrt{|\xi|}\norm{N}_{\Omega_3}.
\end{split}
\end{equation*}
Therefore, the lemma follows.
\end{proof}

\begin{lemma}
Let $~\mathcal{U}[N](\xi)=\frac{1}{\epsilon^2}g_1(\xi)\int_{-\infty}^\xi
N(t)g_2(t)dt$, then $\norm{\mathcal{U}[\bar{M}]}_{\Omega_3}\leq C\delta_2 \norm{U}_{\Omega_3}$ and $\norm{\mathcal{U}[M]}_{\Omega_3}\leq C\frac{\epsilon^2}{|\nu|^2}\norm{U}_{\Omega_3}$.
\end{lemma}
\begin{proof}
(1) For $\xi\in r_3\cap \{ |\xi|\geq R_\epsilon\}$, applying Lemma 2.50 and Lemma 2.48 with $N=\bar{M}$ and $\mathcal{P}(-\infty,\xi)=r_3\cap \{~|\xi|\geq R_\epsilon\}$, we have
\begin{equation}\label{2.61}
\left\vert \mathcal{U}[\bar{M}]\right\vert \leq C\frac{\delta_2}{\epsilon^2} R_\epsilon^{-\tau}|\xi-2i|^{-\tau}\norm{U}_{\Omega_3}.
\end{equation}
(2) For $\xi\in r_3\cap \{ R\leq |\xi|\leq R_\epsilon\}$, applying Lemma 2.51 and Lemma  2.48 with $N=\bar{M}$ and $\mathcal{P}(R^l_\epsilon,\xi)=r_3\cap \{ R_\epsilon \ge ~|\xi| \geq R\}$, we have
\begin{equation}\label{2.62}
\left\vert \mathcal{U}[\bar M]\right\vert \leq C\left[\frac{\delta_2}{\epsilon^2} R_\epsilon^{-\tau}+R_\epsilon^{1-\tau}\delta_2\right]|\xi-2i|^{-\tau}\norm{U}_{\Omega_3}.
\end{equation}
(3) For $\xi\in r_3\cap \{ ~|\xi|\leq R\}$, applying Lemma 2.52, Lemma 2.53 and Lemma 2.48 with $N=\bar{M}$ and $\mathcal{P}(R^l,\xi)=r_3\cap \{ R\ge ~|\xi|\geq 0\}$, we have
\begin{equation}\label{2.63}
\left\vert \mathcal{U}[\bar M]\right\vert \leq C\left[\frac{\delta_2}{\epsilon^2} R_\epsilon^{-\tau}+R_\epsilon^{1-\tau}\delta_2+\delta_2\right]|\xi-2i|^{-\tau}\norm{U}\leq C\delta_2\norm{U}_{\Omega_3}, \text { for } R_\epsilon=O(\epsilon^{-2}).
\end{equation}
The first part of the lemma follows from (\ref{2.61})-(\ref{2.63}), the second part can be proved similarly.

\end{proof}

{\em Proof of Theorem \ref{thm:2.41}:}
From (\ref{2.48}),
Lemmas  2.47-2.54, we obtain
\begin{equation}
\label{2.64}
\norm{U}^-_{\Omega_3}\leq  (\norm{\mathcal{U}[\bar{M}]}_{\Omega_3}+\norm{\mathcal{U}[M]}_{\Omega_3} )
~\le~ C (\delta_2+\frac{\epsilon^2}{|\nu|^2})\norm{U}^-_{\Omega_3} ;
\end{equation}
where $C$ is some constant independent of $\epsilon$, $\nu$ and $\delta_2$.
From (\ref{2.56}), when $\norm{F}_0$  are small enough,
$C(\delta_2+\frac{\epsilon^2}{|\nu|^2}) <1$ in (\ref{2.64}).
This implies $U(\xi)\equiv 0$ on $\Omega_3$ and hence everywhere by
analytic continuation.
Hence $F(\xi)=F_+(\xi)=\hat F (\xi)$ and
$F(\xi)$ is analytic in the upper half plane. Thus
for
$\xi \in \left \{ ~\text{Im ~}\xi < 0 \right \} \cap \mathcal{R}^-$,
$I_2 (\xi) =I_1 (\xi)$,
 and equation (\ref{2.34})
reduces to (\ref{2.27}) in that region.

\subsection{Formulation of the Half Problem}
\label{sec:2.3}

If $F \in \mathbf{A}_0^-$ and satisfies symmetry condition
Im ~$F = 0$ on
$\left \{ ~\text{Re ~}\xi ~=0 \right \} \cap \mathcal{R}$, then
the Schwartz reflection principle applies and
\begin{equation}
\label{2.65}
F(\xi)=[F(-\xi^*)]^*~ \quad \text{for }\xi\in\mathcal{R};
\end{equation}
defines $F$ in $\mathcal{R}^+$; consequently $F \in \mathbf{A}_0$
with $ \norm{F}_0 = \norm{F}_0^-$
The reflection principle also implies
\begin{equation}
\label{2.66}
F'(\xi)=-[F'(-\xi^*)]^* ~\quad \text{ for }\xi\in\mathcal{R}.
\end{equation}

For $F\in\mathbf{A}^-$,
if we relax the symmetry condition
Im ~$F = 0$ on
$\left \{~\text{Re ~}\xi ~=0 \right \} \cap \mathcal{R}$, then
it is still possible to define $F$ and its derivative
in $\mathcal{R}^+$, based on $F$ in $\mathcal{R}^-$ using
(\ref{2.65})-(\ref{2.66}). However, this $F$ in
$\mathcal{R}^+$ is not the analytic continuation of
$F$ in $\mathcal{R}^-$ since violation of the symmetry condition
implies that extension of $F$ in $\mathcal{R}^+$ is
discontinuous across $\{ ~\text{Re ~}\xi = 0 \} \cup \mathcal{R}$.
Nonetheless, this still allows us to define
analytic functions $F_-$ in $\Omega_0$ through (\ref{2.18}), and
$F_+$
in $\Omega_1$ through (\ref{2.30}),
each of which have vanishing imaginary parts
on the Im ~$\xi$ axis segment that are part of their domains of
analyticity.
Thus, $I_2$ is still defined as in (\ref{2.33}) as an
analytic function everywhere in $\mathcal{R}$.
Also, the norms of these
functions $F_-$, $F_+$,  and
$I_2$ in their
respective domains are completely
controlled by $\|F\|_0^-$. \\

\noindent{\bf Half Problem:}
{\em Find function
$F\in\mathbf{A}^-_{0,\delta}$ that is analytic in
$\mathcal{R}^-$ and continuous in its closure,
and satisfies equation (\ref{2.34}) in
$\mathcal{R}^-$.}

\begin{remark}\label{thm:2.49}
If $F\in\mathbf{A}^-_\delta$
is a solution of the Half Problem,
and $F$ satisfies
\begin{equation}
\label{2.67}
\text{Im~ }F=0 ,\text{ on }\{~\text{Re }\xi=0\}\cap\mathcal{R};
\end{equation}
then $F$ is a solution to {\bf Problem 2} and therefore the
original Finger Problem.
Conversely, any solution $F$ to
{\bf Problem 2} (and therefore the original Finger Problem) is also a solution
to the Half Problem.
\end{remark}

\section{Solution to the Half Problem  in $\mathcal{R}^-$}

In this section, by changes of variables, we first analyze equation (\ref{2.34}) and identify the possible singularity at $\xi=0$ which corresponds to the finger tip.
We then formulate an integral equation that is equivalent to (\ref{2.34}). Near the singular point $\xi =0$, we seek the particular solution satisfying (\ref{3.33}) and derive the differential and integral equations that govern the particular solution. By constructing a normal approximation sequence, we obtain the existence of solution to the Half Problem.

\subsection{ Analysis of Equation (\ref{2.34}) }

Let $F=\epsilon^2 W$, then (\ref{2.34}) becomes
\begin{equation}\label{3.1}
\epsilon^2 W'(\xi)=\frac{1}{2}\big\{-\big[ (W+ I_2)^2((F^-)'+\bar H)+2 G_1( F^-)\big] +((F^-)'+\bar H)\sqrt{(W+I_2)^2-4}\big\}
\end{equation}
Let
\begin{equation}\label{3.2}
Q(\xi)=-\frac{1}{2\pi }\int_{-\infty}^{\infty}
\frac{(H-\bar H)(t)dt}{H^{1/2}(t){\bar H}^{1/2}(t)(t-\xi)} ~\qquad \text{ for Im ~}\xi ~<~0,
\end{equation}
then
\begin{equation}\label{3.3}
\begin{split}
{\tilde I}(F) &=I_2(G)-Q(\xi)\\
&=-\frac{1}{2\pi }\int_{r_2}
\frac{(F_+'-F_-') dt}{G_5(F_+,F_-)(t-\xi)}\\
& -\frac{1}{2\pi }\int_{r_2}\frac{(H-\bar H)}{H^{1/2}\bar H^{1/2}}
\frac{(-\epsilon^2F_+'F_-'-F_+'\bar H-F_-'H) dt}{G_5(F_+,F_-)[H^{1/2}\bar H^{1/2}+G_5(F_+,F_-)](t-\xi)},
\end{split}
\end{equation}
\begin{equation}\label{3.4}
G_5(F_+,F_-)=(F_+'+H)^{1/2}(F_-'+\bar H)^{1/2}.
\end{equation}
Now let
\begin{equation}\label{3.5}
V=W+Q,
\end{equation}
then (\ref{3.1}) becomes
\begin{equation}\label{3.6}
\epsilon^2 V'=\frac{-[V^2\bar H +2(H-\bar H)]+\bar H V\sqrt{V^2-4}}{2}+\epsilon^2Q'+ G_6(V,F_-,\tilde I),
+G_7(V,\tilde I)
\end{equation}
where
\begin{equation}\label{3.7}
G_6(V,F^-,\tilde I)=F_-'\big[\frac{-(V+\tilde I)^2+(V+\tilde I)\sqrt{(V+\tilde I)^2-4}+2}{2}\big],
\end{equation}
\begin{equation}\label{3.8}
G_7(V,\tilde I)=-\frac{[2V\tilde I+(\tilde I)^2]\bar H}{2}+\frac{1}{2}\bar H\frac{[2V\tilde I+(\tilde I)^2][(V^2-4)+(V+\tilde I)^2]}{(V+\tilde I)\sqrt{(V+\tilde I)^2-4}+V\sqrt{V^2-4}}.
\end{equation}
Let
\begin{equation}\label{3.9}
V_0=-\frac{2\gamma}{\sqrt{\xi^2+\gamma^2}},
\end{equation}
then
\begin{equation*}\label{3.10}
\sqrt{V_0^2-4}=-\frac{2i\xi}{\sqrt{\xi^2+\gamma^2}},
\end{equation*}
and
\begin{equation*}\label{3.11}
-[V_0^2\bar H +2(H-\bar H)]+\bar H V_0\sqrt{V_0^2-4}=0.
\end{equation*}
Let $p=V-V_0$, then $p$ satisfies
\begin{equation}\label{3.12}
\epsilon^2 p'+\tilde Q p=\epsilon^2(Q'-V_0')+ G_6(p,\tilde I)
+G_7(p,\tilde I)+G_8(p),
\end{equation}

where
\begin{equation}\label{3.13}
\tilde Q(\xi)=\frac{2iH^{3/2}\bar H^{1/2}}{H+\bar H}=i\frac{(\xi+i\gamma)^{3/2}(\xi-i\gamma)^{1/2}}{\xi (\xi^2+1)},
\end{equation}
\begin{multline}\label{3.14}
G_8(p)=\frac{\frac{1}{2} \bar H p^2[ (2V_0+p)^2+(2V_0^2-4)]}{(p+V_0)\sqrt{V_0^2-4+(2V_0p+p^2)}+V_0\sqrt{V_0^2-4}}-\frac{\bar H p^2}{2}\\
-\frac{p^2(V_0^2-2)(2V_0+p)[p^2+2V_0p+2V_0^2-4]}{\left((p+V_0)\sqrt{V_0^2-4+(2V_0p+p^2)}+V_0\sqrt{V_0^2-4}\right)^2\sqrt{V_0^2-4}},
\end{multline}
where $G_6(p,\tilde I)$ and $G_7(p,\tilde I)$ are given by (3.7) and (3.8) with $V$ being replaced by $p+V_0$.
\begin{remark}
It is to be noted that $\tilde Q$ is singular at $\xi=0$ which corresponds to the finger tip.
\end{remark}

\begin{lemma}
$F$ is a solution to the half problem  in $\mathcal{R}^-$ if and only if $p=\frac{1}{\epsilon^2}F+Q-V_0$ is a solution to (\ref{3.12}) in $\mathcal{R}^-$.
\end{lemma}
\begin{proof}
The lemma follows from that (\ref{2.34}) is equivalent to (\ref{3.12}).
\end{proof}

Let
\begin{equation}\label{3.15}
N(p)=\epsilon^2(Q'-V_0')+ G_6(p)+G_7(p)+G_8(p).
\end{equation}

\begin{lemma}
$p\in \mathbf{A}_0$ satisfies (\ref{3.12}) if and only if  $p$ satisfies the following integral solution
\begin{equation}\label{3.16}
p(\xi)=\mathcal{U}(N)[\xi]\equiv \frac{1}{\epsilon^2}g_1(\xi)\int_{-\infty}^\xi g_2(t) N(p)(t) dt.
\end{equation}
\end{lemma}
\begin{proof}
Since $p$ satisfies (\ref{3.12}), then
\begin{equation*}\label{3.17}
p(\xi)=C g_1(\xi)+\frac{1}{\epsilon^2}g_1(\xi)\int_{-\infty}^\xi g_2(t) N(p)(t) dt
\end{equation*}
for some constant $C$. Since $p(-\infty)=0, g_1(-\infty)\ne 0$, this implies that $C=0$.
\end{proof}

In this section, we choose $\tau\geq 6/7, R_\epsilon=O(\epsilon^{-2})>R$.\\

In order to estimate the integral operator in (\ref{3.16}), we need to divide $\mathcal{R}^-$ into subregions so that the integral on each subregion can be controlled using Lemmas 2.50 - 2.53.

Let $\mathcal{R}_1=\mathcal{R}^-\cap \{\xi: |\xi| \leq R_\epsilon\}.$

\begin{lemma}
Assume that $p, f$ and $g\in\mathbf{A}^-_{0,\tilde\delta}$ .
Then for $\xi\in\mathcal{R}_1$
\begin{equation}\label{3.18}
|{G}_k(p) [\xi]|
\leq C\epsilon^2 (1+\norm{p}_0)\left(1+ |\xi-2i|^{-\tau}\right)|\xi-2i|^{-1-\tau}, k=6,7,
\end{equation}
\begin{equation}\label{3.19}
|{G}_k(f)-G_k(g) [\xi]|
\leq C|\xi-2i|^{-1-\tau} (\epsilon^2+\tilde\delta \epsilon^2|\xi-2i|^{-\tau})\norm{f-g}, k=6,7,
\end{equation}
\begin{equation}\label{3.20}
|{G}_8(p) [\xi]|
\leq C |\xi-2i|^{-1-2\tau} \norm{p}^2,
\end{equation}
\begin{equation}\label{3.21}
|{G}_8(f)-G_8(g) [\xi]|
\leq C\tilde\delta |\xi-2i|^{-1-2\tau}\norm{f-g},
\end{equation}
where $C$ is a constant independent of $\epsilon $.
\end{lemma}
\begin{proof}
(\ref{3.18}) and (\ref{3.19}) follow from (\ref{3.7}) and (\ref{3.8}); (\ref{3.20}) and (\ref{3.21}) follow from (\ref{3.14}).
\end{proof}
\begin{lemma}
Assume that $f,g\in\mathbf{A}^-_{0,\tilde\delta}$ .
Then for $\xi\in\mathcal{R}_1$
\begin{equation}\label{3.22}
|\mathcal{U}(f) [\xi]|
\leq C (R_\epsilon^{-1+\tau}+\frac{ \tilde\delta^2}{R_\epsilon^{\tau}\epsilon^2})|\xi-2i|^{-\tau},
\end{equation}
\begin{equation}\label{3.23}
|\mathcal{U}(f) [\xi]-\mathcal{U}(f) [\xi]|
\leq C(\frac{1}{R_\epsilon^{1-\tau}}+\frac{ \tilde\delta}{R_\epsilon^{\tau}\epsilon^2})\norm{f-g}_0|\xi-2i|^{-\tau},
\end{equation}
where $C$ is a constant independent of $\epsilon $.
\end{lemma}
\begin{proof}
(\ref{3.22}) and (\ref{3.23}) follow from Lemma 2.50, Lemma 3.3 and Lemma 3.4.
\end{proof}

Let $\mathcal{R}_2=\mathcal{R}^-\cap \{\xi: R\le| \xi|\le R_\epsilon\}
$ and $R^l_\epsilon=r^-_l\cap \{\xi: |\xi|=R_\epsilon \}$. From Lemma 6.3 and Lemma 6.9 in the appendix, Re $P(\xi)$ attain minimum at $R^l_\epsilon$.

Define operator $\mathcal{U}_2$ as
\begin{equation}\label{3.24}
\mathcal{U}_2[p]=\frac{1}{\epsilon^2}g_1(\xi)\int_{R^l_\epsilon}^\xi g_2(t) N(p)(t) dt.
\end{equation}

\begin{lemma}
Assume that $f,g\in\mathbf{A}^-_{0,\tilde\delta}$ . then for $\xi\in \mathcal{R}_2$
\begin{equation}\label{3.25}
\begin{split}
&|N(f)(\xi)|\leq C\epsilon^2\left(|\xi|^{-2}+ |\xi|^{-(1+\tau)}\norm{f}_0\right)+C\norm{f}^2_0 |\xi|^{-1-2\tau},\\
&|N(f)(t)-N(g)(t)|
\leq C\epsilon^2 |\xi|^{-1-\tau} \norm{f-g}_0+\left(\norm{g}+\norm{f}_0\right) |\xi|^{-1-2\tau} \norm{f-g}_0,
\end{split}
\end{equation}
\begin{equation}\label{3.26}
|\xi-2i|^\tau |\mathcal{U}_2(f) [\xi]|
\leq C\epsilon^2\left(R_\epsilon^{\tau}+ R_\epsilon^{1}\norm{f}_0\right)+C\norm{f}^2_0 R_\epsilon^{
1-\tau},
\end{equation}
\begin{equation}\label{3.27}
|\xi-2i|^\tau |\mathcal{U}_2(f) [\xi]-\mathcal{U}_2(f) [\xi]|
\leq C\epsilon^2 R_\epsilon^{(1-\tau)} \norm{f-g}_0+C(\norm{f}_0+\norm{g}_0)R_\epsilon^{
1-\tau} \norm{f-g}_0,
\end{equation}
where $C$ is a constant independent of $\epsilon $.
\end{lemma}
\begin{proof}
(\ref{3.25}) follows from (\ref{3.7}),(\ref{3.8} ),(\ref{3.14}) and (\ref{3.15}); (\ref{3.26}) and (\ref{3.27}) follow from Lemma 2.51 and (\ref{3.25}).
\end{proof}

Let $\mathcal{R}_3=\{
|\xi|\le R\}\cap \mathcal{R}^-$ and $R^l=r_l^-\cap \{\xi : |\xi|=R\}$.\\

Define operator $\mathcal{U}_3$ as
\begin{equation}\label{3.28}
\mathcal{U}_3[f]=\frac{1}{\epsilon^2}g_1(\xi)\int_{R^l}^\xi g_2(t) N(f)(t) dt.
\end{equation}

\begin{lemma}
Assume that $f\in\mathbf{A}^-_{0,\tilde\delta}$ .
Then for $\xi\in\mathcal{R}_3$
\begin{equation}\label{3.29}
\begin{split}
&|N(f)(\xi)|\leq C\epsilon^2+C\norm{f}^2_0,\\
&|N(f)(\xi)-N(g)(\xi)|
\leq C\epsilon^2  \norm{f-g}_0+\left(\norm{g}+\norm{f}_0\right)  \norm{f-g}_0,
\end{split}
\end{equation}
\begin{equation}\label{3.30}
|\mathcal{U}_3(f) [\xi]|
\leq C\epsilon^2+C\norm{f}^2_0 ,
\end{equation}
\begin{equation}\label{3.31}
|\mathcal{U}_3(f) [\xi]-\mathcal{U}_3(f) [\xi]|
\leq C\left(\epsilon^2 +C(\norm{f}_0+\norm{g}_0)\right) \norm{f-g}_0,
\end{equation}
where $C$ is a constant independent of $\epsilon $.
\end{lemma}
\begin{proof}
(\ref{3.29}) follows from (\ref{3.7}), (\ref{3.14}) and (\ref{3.15}); (\ref{3.30}) and (\ref{3.31}) follow from Lemma 2.52, Lemma 3.3 and (\ref{3.29}).
\end{proof}

\subsection{ Analysis of (\ref{3.12}) in a neighborhood of the origin}

Let $\mathcal{T}$ be  the region bounded by
$r_{u,3}$, negative imaginary axis, line segment
$\{\xi: \xi=-\nu_1+se^{-\pi i/6}, 0\le s\le
2\sqrt{3}/3 \nu_1 \}$ and line segment $\{\xi: \xi=-\nu_1+se^{\pi i/6}, 0\le s\le
2\sqrt{3}/3 \nu_1 \}$ where $\nu_1$ is as in Definition 1.1. Then $\mathcal{T}$ is a neighborhood of $\xi=0$ in $\mathcal{R}^-$.

 In  $ \mathcal{T}$, we have from (\ref{3.13})
\begin{equation}\label{3.32}
\tilde Q(\xi)=-\frac{\gamma^2}{\xi}+Q_1(\xi),
\end{equation}
where $Q_1(\xi)$ is analytic at $\xi=0$ .
We seek solution of the form
\begin{equation}\label{3.33}
p=(\beta + q(\xi))\xi,
\end{equation}
where $\beta$ is a nonzero constant that will be given later and  $q(\xi)$ satisfies
\begin{equation}\label{3.34}
|q(\xi)|=O(\sqrt{|\xi|}) \text{ as } \xi\to 0.
\end{equation}

We first want to derive the governing differential and integral equations for $q(\xi)$ in the region $\mathcal{T}$.\\

For $\xi\in  \mathcal{T}$, using $F= \epsilon^2 (p+V_0-Q)$ we can write
\begin{equation}\label{3.35}
F'_-(\xi)=F_{0,-}(\xi)+\epsilon^2p'_-(\xi),
\end{equation}
where
\begin{equation}\label{3.36}
\begin{split}
p'_-(\xi)&=-\frac{1}{2\pi i }\int_{r_0^-}\frac{\bar p'(t)-\bar p'(0)}{t-\xi}dt
-\frac{1}{2\pi i }\int_{r_0^-}\frac{[\bar p'(t)]^*-[\bar p'(0)]^*}{t^*+\xi}dt,
\end{split}
\end{equation}
\begin{equation}\label{3.37}
F_{0,-}(\xi)=-\frac{\epsilon^2}{2\pi i }\int_{r_0}\frac{\bar V_0'(t)-\bar Q'(t)}{t-\xi}dt.
\end{equation}
\begin{equation}\label{3.38}
F'_+(\xi)=F_{0,+}(\xi)+\epsilon^2 p'_+(\xi),
\end{equation}
where
\begin{equation}\label{3.39}
p'_+(\xi)=\frac{1}{2\pi i }\int_{r_1^-}\frac{p'(t)- p'(0)}{t-\xi}dt
+\frac{1}{2\pi i }\int_{r_1^-}\frac{[ p'(t)]^*-[ p'(0)]^*}{t^*+\xi}dt,
\end{equation}
\begin{equation}\label{3.40}
F_{0,+}(\xi)=\frac{\epsilon^2}{2\pi i }\int_{r_1}\frac{ V_0'(t)- Q'(t)}{t-\xi}dt.
\end{equation}
Let
\begin{equation}\label{3.41}
G_{5,0}(\xi)=(F_{+,0}'+H)^{1/2}(F_{-,0}'+\bar H)^{1/2},
\end{equation}
\begin{equation}\label{3.42}
G_{5,1}(p_-',p_+')=G_{5}(F'_-,F_+')-G_{5,0}(\xi).
\end{equation}
Let
\begin{equation}\label{3.43}
\begin{split}
\tilde G(F_-',F_+')&=
\frac{(F_+'-F_-') }{G_5(F_+,F_-)}\\
& -\frac{(H-\bar H)}{H^{1/2}\bar H^{1/2}}
\frac{(-\epsilon^2F_+'F_-'-F_+'\bar H-F_-'H) }{G_5(F_+,F_-)[H^{1/2}\bar H^{1/2}+G_5(F_+,F_-)]},
\end{split}
\end{equation}

\begin{equation}\label{3.44}
\tilde G_1(p_-',p_+')=\tilde G(F_-',F_+')-\tilde G(F_{-,0}',F_{+,0}'),
\end{equation}
and
\begin{equation}\label{3.45}
{\tilde I}(F)
=-\frac{1}{2\pi }\int_{r_2}
\frac{\left(\tilde G_1(p_-',p_+')[t]-\tilde G_1(p_-',p_+')[0]\right) dt}{(t-\xi)}-\frac{1}{2\pi }\int_{r_2}\frac{ \tilde G(F_{-,0}',F_{+,0}')dt}{(t-\xi)}.
\end{equation}
Let
\begin{equation}\label{3.46}
{\tilde I}_1(p)(\xi) ={\tilde I}(F)(\xi)-{\tilde I}(F)(0).
\end{equation}

Plugging (\ref{3.33}) into (\ref{3.7}), (\ref{3.8}) and (\ref{3.14}), we have for $\xi\in B_\nu \cap \mathcal{R}^-$,
\begin{equation}\label{3.47}
G_6(q,\xi)=\beta_{6,0}(p)+ [F_-'(\xi)-F_-'(0)][G_{6,0}(\xi)+G_{6,1}(\xi,{\tilde I}_1,q)],
\end{equation}
where
\begin{equation*}\label{3.48}
\beta_{6,0}(p)=[F_{-0}'(0)+\epsilon^2 p_-'(0)]\Biggl\{1+\frac{1}{2}\biggl[-(-2+\tilde I(0))^2+(-2+\tilde I(0))\sqrt{-4\tilde I(0)+(\tilde I(0))^2}\biggr]\Biggr\}.
\end{equation*}

\begin{equation}\label{3.49}
G_7(q,\xi)=\beta_{7,0}(p)+\tilde I_1 \big\{G_{7,0}(\xi)+G_{7,1}(\xi,{\tilde I}_1,q\big\}) ,
\end{equation}
where
\begin{equation}\label{3.50}
\beta_{7,0}(p)=-\frac{i\gamma}{2}[-4\tilde I(0)+(\tilde I(0))^2]-\frac{i\gamma}{2}(-2+\tilde I(0))\sqrt{-4\tilde I(0)+(\tilde I(0))^2}.
\end{equation}
\begin{equation}\label{3.51}
G_8(q)=-i\gamma \beta -i\gamma q+\sqrt{\xi} G_{8,1}(\sqrt{\xi},q)+q^2 G_{8,2}(\sqrt{\xi},q),
\end{equation}
 where $G_{6,1}(\xi,{\tilde I}_1, q), G_{7,1}(\xi,{\tilde I}_1, q)$ are analytic in $\xi,{\tilde I}_1$ and $q$. $G_{8,1}(\xi,q),G_{8,2}(\xi,q)$ are analytic at $\xi=0,q=0$.\\
Let
\begin{equation}\label{3.52}
\beta(p)=\frac{\epsilon^2(Q'(0)-V_0'(0))+ \beta_{6,0}(p)+\beta_{7,0}(p)}{\epsilon^2-\gamma^2+i\gamma}.
\end{equation}
Equation (\ref{3.12}) becomes
\begin{equation}\label{3.53}
\epsilon^2 q'-\frac{\gamma^2-\epsilon^2-i\gamma}{\xi} q=N_1(\xi,q),
\end{equation}
where
\begin{multline}\label{3.54}
N_1(\xi,p)=\epsilon^2 Q_2(\xi)-Q_1(\xi) \beta-Q_1(\xi) q+
\frac{F_-'(\xi)-F_-'(0)}{\xi}\left\{G_{6,0}(\xi)+G_{6,1}(\xi, \tilde I_1,q)\right\}\\
+\frac{\tilde{I}_1}{\xi}\left\{G_{7,0}(\xi)+G_{7,1}(\xi, \tilde I_1,q)\right\}
+\frac{1}{\sqrt{\xi}}G_{8,1}(\sqrt{\xi},q)+\frac{q^2}{\xi} G_{8,2}(\sqrt{\xi},q),
\end{multline}
where $Q_2(\xi)$ is given by
\begin{equation*}\label{3.55}
Q_2(\xi)= \frac{(Q'(\xi)-V'_0(\xi))-(Q'(0)-V'_0(0))}{\xi}.
\end{equation*}
Let
\begin{equation}\label{3.56}
h_1(\xi)=e^{-\frac{P_1(\xi)}{\epsilon^2}}, ~~h_2(\xi)=e^{\frac{P_1(\xi)}{\epsilon^2}}.
\end{equation}
where $P_1(\xi)$ is
\begin{equation}\label{3.57}
P_1(\xi)=(\gamma^2-\epsilon^2-i\gamma)\log \xi
\end{equation}
and
\begin{equation*}
\log~ \xi=\ln |\xi|+\arg \xi, ~ \pi\leq\arg \xi\leq \frac{3\pi}{2}.
\end{equation*}
\begin{remark}
Re $P_1(\xi)$ attains maximum in $\mathcal{T}$ at $\xi=-i\frac{\sqrt{3}}{3}\nu_1$. For any $\xi\in \mathcal{T}$, there is a path $\mathcal{P}(\xi,-i\frac{\sqrt{3}}{3}\nu_1)$ from $\xi$ to $-i\frac{\sqrt{3}}{3}\nu_1$ such that Re $P_1(\xi)$ decreases from $\xi$ to $-i\frac{\sqrt{3}}{3}\nu_1$ and $\vert\frac{d Re ~P_1(\xi)}{d\xi}\vert\geq \frac{C}{|\xi|}$.
\end{remark}

The integral form of (\ref{3.53}) is
\begin{equation}\label{3.58}
q=\tilde{q}(-i\frac{\sqrt{3}}{3}\nu_1)h_1(-i\frac{\sqrt{3}}{3}\nu_1)h_2(\xi)+h_2(\xi)\int_{-i\frac{\sqrt{3}}{3}\nu_1}^\xi h_1(t)\frac{N_1(t,q(t))}{\epsilon^2}dt,
\end{equation}
where
\begin{equation*}\label{3.59}
\begin{split}
\tilde{q}(-i\frac{\sqrt{3}}{3}\nu_1)&=\frac{p(-i\frac{\sqrt{3}}{3}\nu_1)}{-i\frac{\sqrt{3}}{3}\nu_1}-\beta,\\
p(-i\frac{\sqrt{3}}{3}\nu_1)&=\mathcal{U}(N)[-i\frac{\sqrt{3}}{3}\nu_1]\equiv \frac{1}{\epsilon^2}g_1(-i\frac{\sqrt{3}}{3}\nu_1)\int_{-\infty}^{-i\frac{\sqrt{3}}{3}\nu_1} g_2(t) N(p)(t) dt.
\end{split}
\end{equation*}

We seek solution $q(\xi)$ to (\ref{3.58}) in the function  space $\mathbf{D}$, which is defined as
\begin{equation*}\label{3.60}
\begin{split}
\mathbf{D}=\{ q(\xi):~ &q \text{ is analytic in $B_\nu\cap\mathcal{R}^- $ and bounded in its closure  with}\\
&\underset{\xi\in\overline{B_\nu\cap\mathcal{R}^-}}{\sup}~\abs{\xi^{-1/2} q(\xi)}<\infty~\}.
\end{split}
\end{equation*}
with norm
\begin{equation*}
\norm{q}_{\mathbf{D}}=\underset{\xi\in\overline{B_\nu\cap\mathcal{R}^-}}{\sup}~\abs{\xi^{-1/2} q(\xi)}.
\end{equation*}

The following lemmas give estimates of $N_1$ and the integral operator in (\ref{3.58}) in the norm of $\mathbf{D}$.

\begin{lemma}\label{L4.2}
If $ \xi^{1/2} n(\xi)\in C(\overline{B_\nu\cap\mathcal{R}^-})$, let $ q_1 (\xi)=
\frac{1}{\epsilon^2}h_2(\xi)\int_{-i\frac{\sqrt{3}}{3}\nu_1}^{\xi}
n(t)g_1(t)dt$, then $ q_1\in \mathbf{D}$ and
$\norm{ q_1}\leq C \max |\xi^{1/2}n(\xi)|$ for constant $C>0$ independent of
$\epsilon$ and $\nu$.
\end{lemma}
\begin{proof}
We choose path $\mathcal{P}(\xi,-i\frac{\sqrt{3}}{3}\nu_1)$ as in Remark 3.8,
\begin{equation*}
\begin{split}
|q_1(\xi)|&=\left\lvert\frac{1}{\epsilon^2}\int_{-i\frac{\sqrt{3}}{3}\nu_1}^\xi n(t)\exp \{-\frac{\gamma^2-\epsilon^2-i\gamma}{\epsilon^2}(\log t-\log \xi)\}dt\right\rvert\\
&\leq C \max|t^{1/2}n(t)|\int_{-i\frac{\sqrt{3}}{3}\nu_1}^\xi\left\lvert t^{-1/2}\exp \{-\frac{\gamma^2-\epsilon^2-i\gamma}{\epsilon^2}(\log t-\log \xi)\}\right\rvert|dt| \\
&\leq C |\xi|^{1/2}\max|t^{1/2}n(t)|.
\end{split}
\end{equation*}
\end{proof}

\begin{lemma}
Let $q\in \mathbf{D}$, then $\max|\xi|^{1/2}| N_1(\xi,q(\xi))|\leq C_1+C_2(\delta_1+|\nu|)\norm{q}_{\mathbf{D}}$,
where constants $C_1$ and $ C_2$ are independent of
$\epsilon$ and $\nu$.
\end{lemma}
\begin{proof}
The lemma follows from (\ref{3.47})-(\ref{3.51}) and (\ref{3.54}).
\end{proof}
\begin{lemma}
Let $q_1,q_2\in \mathbf{D}$, then $\max|\xi|^{1/2}| N_1(\xi,q_1(\xi))-N_1(\xi,q_2(\xi))|\leq C(\delta_1+|\nu|)\norm{q_1-q_2}_{\mathbf{D}}$,
constant $C$ is independent of
$\epsilon$ and $\nu$.
\end{lemma}
\begin{proof}
The lemma follows from (\ref{3.47})-(\ref{3.51}) and (\ref{3.54}).
\end{proof}

\subsection{Existence of solution to (\ref{3.12})}

In this section, by constructing a normal sequence, we are going to prove the existence of solution $p(\xi)$ to (\ref{3.12}) which satisfies (\ref{3.33}) in $\mathcal{T}$.\\

From Definitions ( \ref{3.16}), ( \ref{3.24}) and (\ref{3.58} ), we can write
\begin{equation}\label{3.61}
\mathcal{U}[p](\xi)=g_1(\xi)\left[g_2(R^l_\epsilon)\mathcal{U}(R^l_\epsilon)+g_2(\xi)\mathcal{U}_2[p](\xi)\right] \quad \text{ for } \xi\in \mathcal{R}_2,
\end{equation}
\begin{equation}\label{3.62}
\mathcal{U}[p](\xi)=g_1(\xi)\left[g_2(R^l_\epsilon)\mathcal{U}(R^l_\epsilon)+g_2(R^l)\mathcal{U}_2[p](R^l)+g_2(\xi)\mathcal{U}_3[p](\xi)\right] \quad \text{ for } \xi\in \mathcal{R}_3/\mathcal{T},
\end{equation}
and
\begin{equation}\label{3.63}
\mathcal{U}[p](\xi)=\xi (\beta [p]+\mathcal{U}_4[q](\xi))
\text{ for } \xi\in \mathcal{T},
\end{equation}
where
\begin{multline}\label{3.64}
\mathcal{U}_4[q](\xi)=\left(-\beta[p]+\frac{\mathcal{U}(-i\frac{\sqrt{3}}{3}\nu_1)}{-i\frac{\sqrt{3}}{3}\nu_1}\right)h_1(-i\frac{\sqrt{3}}{3}\nu_1)h_2(\xi)\\
+h_2(\xi)\int_{-i\frac{\sqrt{3}}{3}\nu_1}^\xi h_1(t)\frac{N_1(t,q(t))}{\epsilon^2}dt,
\end{multline}
and $\beta$ is given by (\ref{3.52}).\\

In the following, we choose $\nu=O(\epsilon)<\nu_1.$\\

We first obtain some estimates of $\beta$ and $\tilde{I}(p)$ in the small neighborhood of $\xi=0$.

\begin{lemma}
Let $p\in \mathbf{A}^-_0$, then for $\xi\in B_\nu\cap\mathcal{R}^-$
\begin{equation*}\label{3.65}
\frac{\vert\tilde G_1(p'_-,p_+')[\xi]-\tilde G_1(p'_-,p_+')[0]\vert}{\sqrt{|\xi|}}\leq \epsilon^2 \left( K \underset{\xi\in \overline{B_\nu\cap\mathcal{R}^-}}{\sup}~\left\vert\frac{p'(\xi)-p'(0)}{\sqrt{|\xi|}}\right\vert+C\right),
\end{equation*}
and
\begin{equation*}\label{3.66}
\underset{\xi\in r_2\cap B_\nu}{\sup}~|\tilde G_1(p'_-,p_+')|\leq \epsilon^2 \left(\frac{\norm{p}_0}{\nu}+C\right).
\end{equation*}
$C>0$ is a constant independent of $\epsilon$ and $\nu$.
\end{lemma}
\begin{proof}
The lemma follows from (\ref{3.43}) and (\ref{3.44}).
\end{proof}
\begin{lemma}
Let $p\in \mathbf{A}^-_0$, then for $\xi\in B_\nu\cap\mathcal{R}^-$
\begin{equation*}\label{3.67}
\frac{\vert\tilde I(p)[\xi]-\tilde I(p)[0]\vert}{\sqrt{|\xi|}}\leq \epsilon^2 C\left( \underset{\xi\in \overline{B_\nu\cap\mathcal{R}^-}}{\sup}~\left\vert\frac{p'(\xi)-p'(0)}{\sqrt{|\xi|}}\right\vert+\nu^{-1/2}+\nu^{-3/2}\underset{\xi\in \mathcal{R}_-/B_\nu}{\sup}|p|\right)
\end{equation*}
and
\begin{equation*}\label{3.68}
C\epsilon^2\leq\vert \tilde I(p)[0]\vert \leq \epsilon^2 C\left( \sqrt{\nu}\underset{\xi\in \overline{B_\nu\cap\mathcal{R}^-}}{\sup}~\left\vert\frac{p'(\xi)-p'(0)}{\sqrt{|\xi|}}\right\vert+1+\nu^{-1/2}|\log\nu|\underset{\xi\in \mathcal{R}_-/B_\nu}{\sup}|p|\right).
\end{equation*}
$C>0$ is a constant independent of $\epsilon$ and $\nu$.
\end{lemma}

\begin{proof}
The lemma follows from Lemma 2.15, Lemma 2.31 and (\ref{3.45}).
\end{proof}
\begin{lemma}
Let $p\in \mathbf{A}^-_0$, then
\begin{equation*}\label{3.69}
C_1\epsilon\leq |\beta(p)|\leq C\epsilon \left( 1+\left(\sqrt{\nu}\underset{\xi\in\overline{ B_\nu\cap\mathcal{R}^-}}{\sup}~\left\vert\frac{p'(\xi)-p'(0)}{\sqrt{|\xi|}}\right\vert+ \frac{|\log \nu|}{\nu^{1/2}}\underset{\xi\in \mathcal{R}_-/B_\nu}{\sup}|p|\right)^{1/2}\right),
\end{equation*}
where $C_1, C>0$ are constants independent of $\epsilon$ and $\nu$.
\end{lemma}
\begin{proof}
The lemma follows from (\ref{3.52}) and the above lemma.
\end{proof}

Let
\begin{equation}\label{3.70}
\beta_0=\beta(0),~\qquad p_0(\xi)=\frac{\beta_0\xi}{\xi^2+1},~\qquad q_0(\xi)=\frac{\beta_0\xi^2}{\xi^2+1}.
\end{equation}
We define sequences $\{\beta_n\},\{q_n(\xi)\}$ and $\{p_n(\xi)\}$ as follows:
\begin{equation}\label{3.71}
\beta_n=\beta(p_{n-1}),
\end{equation}
\begin{equation}\label{3.72}
\begin{split}
q_n(\xi)&=\mathcal{U}_4[q_{n-1}]\\
&=\left(-\beta_{n-1}+\frac{\mathcal{U}[p_{n-1}](-i\frac{\sqrt{3}}{3}\nu_1)}{-i\frac{\sqrt{3}}{3}\nu_1}\right)h_1(-i\frac{\sqrt{3}}{3}\nu_1)h_2(\xi)\\
& ~~~~~~~+h_2(\xi)\int_{-i\frac{\sqrt{3}}{3}\nu_1}^\xi h_1(t)\frac{N_1(t,q_{n-1}(t))}{\epsilon^2}dt,
\end{split}
\end{equation}

\begin{equation}\label{3.73}
\begin{split}
p_n(\xi)&=\xi (\beta_n+q_n(\xi)) \text{ for } \xi\in \mathcal{T},\\
p_n(\xi)&=\mathcal{U}[p_{n-1}](\xi) \text{ for } \xi\in \mathcal{R}^-/\mathcal{T}.
\end{split}
\end{equation}

\begin{lemma}
Let $h_3(\xi)=\mathcal{U}_4[0]$, then $h_3(\xi)\in \mathbf{D}$.
\end{lemma}
\begin{proof}
The lemma follows from (\ref{3.64}), Lemma 3.7, Lemma 3.9 and Lemma 3.13.
\end{proof}
Let
\begin{equation}\label{3.74}
\delta_2=\norm{\mathcal{U}_4[0]},~ \qquad \delta_3=\underset{\overline{B_\nu\cap\mathcal{R}^-}}{\sup}~ |\xi^{1/2}N_1[0]|.
\end{equation}

\begin{lemma}
For sufficient small $\epsilon$ and $\nu$,  the following holds for all nonnegative integer $n$:
\begin{equation*}\label{3.75}
|\beta_n|\leq C|\nu|^{1/4},
\end{equation*}
\begin{equation*}\label{3.76}
\norm{q_n}_{\mathbf{D}}\leq 2\delta_2,~~\underset{\overline{B_\nu\cap\mathcal{R}^-}}{\sup}~|\xi|^{1/2}|q'_n(\xi)|\leq 2\delta_3 ,
\end{equation*}
\begin{equation*}\label{3.77}
\epsilon^2\underset{\overline{B_\nu\cap\mathcal{R}^-}}{\sup}~\left\vert\frac{p'_n(\xi)-p'_n(0)}{|\xi|^{1/2}}\right\vert\leq 2(\delta_2+\delta_3) ,
\end{equation*}
\begin{equation*}\label{3.78}
\underset{\mathcal{R}^-/B_\nu}{\sup} |(\xi+2i)^\tau p_n|\leq C|\nu|^{1/2},
\end{equation*}
 where $C>0$ is independent of $\epsilon$ and $\nu$ .
\end{lemma}
\begin{proof}
We use induction to prove the lemma. The lemma holds for $n=0$ from (\ref{3.70}). Assume that the lemma holds for all $n\leq k$, then from (\ref{3.71}) and Lemma 3.13, we obtain

\begin{multline*}
|\beta_{k+1}|=|\beta(p_k)|\leq C\epsilon^2+C|\tilde I[p_k](0)|^{1/2}\\
\leq C\epsilon^2+C\left( \epsilon^2\nu^{-1/2}|\log \nu|\underset{\mathcal{R}^-/B_\nu}{\sup}~ |(\xi+2i)^\tau p_k| + \sqrt{\nu}\epsilon^2\underset{\overline{B_\nu\cap\mathcal{R}^-}}{\sup}~\left\vert\frac{p'_k(\xi)-p'_k(0)}{|\xi|^{1/2}}\right\vert\right)^{1/2}\\
\leq C\epsilon^2+\left(C\epsilon^2|\log\nu |+C\sqrt{\nu}(\delta_2+\delta_3)\right)^{1/2}\leq C|\nu|^{1/4}.
\end{multline*}
From (\ref{3.72}), (\ref{3.74}), Lemma 3.9 and Lemma 3.11, we obtain
\begin{multline*}
|\xi|^{-1/2}|q_{k+1}|\leq \left\vert |\xi|^{-1/2}\left(-\beta_{k}+\frac{\mathcal{U}_3[p_{k}](-i\frac{\sqrt{3}}{3}\nu_1)}{-i\frac{\sqrt{3}}{3}\nu_1}\right)h_1(-i\frac{\sqrt{3}}{3}\nu_1)h_2(\xi)\right\vert\\
~~+\left\vert |\xi|^{-1/2}h_2(\xi)\int_{-i\frac{\sqrt{3}}{3}\nu_1}^\xi h_1(t)\frac{N_1(t,q_{k}(t))-N_1(t,0)}{\epsilon^2}dt \right\vert + \norm{\mathcal{U}_4[0]}\\
\leq C|\xi|^{\frac{\gamma^2}{\epsilon^2}}+\norm{\mathcal{U}_4[0]}\leq 2 \delta_2,
\end{multline*}

\begin{multline*}
|\xi|^{1/2}|q'_{k+1}|\leq \left\vert \frac{|\xi|^{-1/2}}{\epsilon^2}\left(-\beta_{k}+\frac{\mathcal{U}_3[p_{k}](-i\frac{\sqrt{3}}{3}\nu_1)}{-i\frac{\sqrt{3}}{3}\nu_1}\right)h_1(-i\frac{\sqrt{3}}{3}\nu_1)h_2(\xi)\right\vert\\
~~~~+\left\vert \frac{|\xi|^{-1/2}}{\epsilon^2}h_2(\xi)\int_{-i\frac{\sqrt{3}}{3}\nu_1}^\xi h_1(t)\frac{N_1(t,q_{k}(t))}{\epsilon^2}dt\right\vert +\left\vert \frac{|\xi|^{1/2}N_1(t,q_{k}(t))}{\epsilon^2}\right\vert \\
\leq \frac{C}{\epsilon^2}|\xi|^{\frac{\gamma^2}{\epsilon^2}}+\frac{\underset{\overline{B_\nu\cap\mathcal{R}^-}}{\sup}~ |\xi^{1/2}N_1[0]|+\norm{\mathcal{U}_4[0]}}{\epsilon^2}\leq \frac{2 \delta_2+2 \delta_3}{\epsilon^2}.
\end{multline*}
From (\ref{3.73}), Lemma 3.9 and Lemma 3.10, we have
\begin{multline*}
\epsilon^2\underset{\overline{B_\nu\cap\mathcal{R}^-}}{\sup}~\left\vert\frac{p'_{k+1}(\xi)-p'_{k+1}(0)}{|\xi|^{1/2}}\right\vert\leq \epsilon^2|\xi|^{-1/2}|q_{k+1}|+\epsilon^2|\xi|^{1/2}|q'_{k+1}|\leq 2(\delta_2+\delta_3).
\end{multline*}
From (\ref{3.73}), (\ref{3.15}), (\ref{3.16}), Lemma 3.5 and Lemma 3.13 we have
\begin{multline*}
\underset{\mathcal{R}^-/B_\nu}{\sup} |(\xi+2i)^\tau p_{k+1}|\leq C\epsilon^2\underset{\mathcal{R}^-/B_\nu}{\sup} |(\xi+2i)^\tau (Q'-V_0')|+\\
C\underset{\mathcal{R}^-/B_\nu}{\sup} |(\xi+2i)^\tau(|G_6(p_k)|+|G_7(p_k)|+|G_8(p_k)|)\\
\leq C\epsilon^2 +C\underset{\mathcal{R}^-/B_\nu}{\sup} |(\xi+2i)^\tau(|F_-'(p_k)|+|\tilde I(p_k)|+|p_k^2|)\\
\leq C\epsilon^2+C \sqrt{\nu}\epsilon^2\underset{\overline{B_\nu\cap\mathcal{R}^-}}{\sup}~\left\vert\frac{p'_k(\xi)-p'_k(0)}{|\xi|^{1/2}}\right\vert\\
+C\epsilon^2|\log|\nu||\underset{\mathcal{R}^-/B_\nu}{\sup} |(\xi+2i)^\tau p_k|+C\nu \leq C|\nu|^{1/2}.
\end{multline*}

\end{proof}

\begin{theorem}
For sufficient small $\epsilon$ and $\nu=O(\epsilon)$, there exist subsequences $\{\beta_{n_k}\}, \{p_{n_k}(\xi)\}$ such that $\lim_{k\to\infty} \beta_{n_k}=\beta, \lim_{k\to\infty} p_{n_k}(\xi)=p(\xi) \text{ in } \mathbf{A}_0^-$, and $p$ is a solution of (\ref{3.12}). Hence $F=\epsilon^2 (p-Q+V_0)$ is a solution of the Half Problem .
\end{theorem}
\begin{proof}
From the above lemma, $\{\beta_n\}$ is a bounded sequence and  $\{p_n(\xi)\}$ is a normal family. From Montel's Theorem, there exist subsequences $\{\beta_{n_k}\}, \{p_{n_k}(\xi)\}$ such that $\lim_{k\to\infty} \beta_{n_k}=\beta, \quad \lim_{k\to\infty} p_{n_k}(\xi)=p(\xi)$ pointwisely. Since $\mathbf{A}_0^-$ is Banach space, $p \in\mathbf{A}_0^-$. The theorem then follows from Lemma 3.2 and Lemma 3.3.
\end{proof}

\section{Selection of Finger Width: Analysis near $\xi = -i$}
\label{sec:4}

\subsection{Derivation of Equation Near $\xi=-i$}
\label{sec:4.1}

In order to investigate whether or not the symmetry condition
Im ~$F = 0$ on $\{~\text{Re ~}\xi =0 \} \cap \mathcal{R}$
is satisfied, it is necessary to investigate a neighborhood of
a turning point ($\xi = -i \gamma$ in our formulation),
as first suggested from formal calculations
of Combescot {\it et al} (1986).
To that effect, we rewrite

\begin{equation}
\label{4.1}
F(\xi)=\epsilon^2I_2(\xi)
+\frac{i\epsilon^2 \left[(F'(\xi)+H)-(F^{\prime}_- (\xi)+\bar{H})\right]}{(F'(\xi)+H)^{1/2}(  F_-'(\xi)+\bar{H})^{1/2}}.
\end{equation}

We introduce
\begin{equation}
\label{4.2}
\alpha = \frac{\gamma-1}{\epsilon^{4/3}} ,
\end{equation}
\begin{equation}
\xi=-i+i2^{1/3}\epsilon^{4/3}y^2,~~G(y)=-i2^{-1/3}\left(i2^{1/3}y^2 F'(\xi)-\frac{1}{2}(2^{1/3}y^2+\alpha)\right)^{-1/2};
\end{equation}
then (\ref{4.1}) becomes:
\begin{equation}
\label{4.4}
\frac{dG}{dy}-\frac{1}{yG^{2}}=-y-\frac{\bar\delta_1}{y}+\epsilon^{2/3}E_2(\epsilon^{2/3},\epsilon^{2/3}y,G,G',y^{-1}),
\end{equation}
where
\begin{equation}
\bar\delta_1=\frac{\alpha}{2^{1/3}},
\end{equation}
and $E_2(\epsilon^{2/3},\epsilon^{2/3}y,G,G',y^{-1})$ is analytic
function of
$\epsilon^{2/3},\epsilon^{2/3}y,G,G',y^{-1}$.

Note that the leading order equation obtained from
dropping $\epsilon$ terms in (\ref{4.4}) is similar to
equation (133)
in Chapman and King \cite{Chapman2}.
In order to get the equation close to the normal
form discussed in Costin \cite{Costin}, it
is convenient to introduce additional change in variables:
\begin{equation}\label{4.6}
\eta=\frac{2}{3}y^{3},~\qquad ~\psi(\eta)=1-yG(y);
\end{equation}
then (\ref{4.4}) becomes:
\begin{multline}
\label{4.11}
\frac{d\psi}{d\eta}+\psi=-\frac{1}{3\eta}-\frac{1}{3\eta}\psi \\
+\frac{\alpha}{6^{2/3}\eta^{2/3}}+\frac{1}{2}\sum_{n=2}^{\infty}(-1)^n(n+1)\psi^n+\epsilon^{2/3}
E_3(\epsilon^{2/3},\epsilon^{2/3}\eta^{2/3},\eta^{-2/3}),
\end{multline}
where
$E_3(\epsilon^{2/3},\epsilon^{2/3}\eta^{2/3},\psi,\eta^{-2/3})$ is analytic in
$\epsilon^{2/3},\epsilon^{2/3}\eta^{2/3},\psi,\eta^{-2/3}$ with a
series representation
convergent for small values of each argument.

It is to be noted
$$ E_3(\epsilon^{2/3},\epsilon^{2/3}\eta^{2/3},\psi,\eta^{-2/3})
= \sum_{m}^{\infty}
E_{m} (\epsilon^{2/3}, \epsilon^{2/3} \eta^{2/3}, \eta^{-2/3}) \psi^m
 $$
and since each of the arguments for $E_{m}$ can be safely be
assumed to be in a compact set, it follows that there exits numbers $A$,
$\rho_2$ are each independent of any parameter so that
\begin{equation}
\label{4.11.8}
|E_{m}| ~<~A \rho_2^{m}.
\end{equation}
\begin{theorem}\label{thm:4.1}
Let $F(\xi)$ be the solution of the Half Problem as in
Theorem 3.17. After change of variables:
\begin{equation}
\label{4.12}
\xi=-i+i\epsilon^{4/3}\frac{3^{2/3}}{2^{1/3}}\eta^{2/3},
\end{equation}
\begin{equation}
\label{4.13}
\psi(\eta,\epsilon,\alpha)=-\left(1+\frac{2^{4/3}\alpha}{3^{2/3}\eta^{2/3}}-2i
F'(\xi (\eta)) \right)^{-1/2}+1,
\end{equation}
$\psi(\eta,\epsilon,\alpha)$ satisfies equation (\ref{4.11})
for $k_0\epsilon^{-2}\leq |\eta|\leq k_1\epsilon^{-2},~ 0\leq\arg \eta<\frac{3\pi}{4}$  ( where
$k_0$ and $k_1$ are some constants
independent of $\epsilon$ ) and the asymptotic condition
\begin{equation}
\label{4.14}
\psi(\eta,\epsilon,\alpha)\to 0,~\quad 0\leq\arg \eta<\frac{3\pi}{4},~as~\epsilon\to 0;
\end{equation}
in that domain.
\end{theorem}
\begin{proof}
Since $\eta=O(\epsilon^{-2})$ in the given domain,
using (\ref{4.12}), we have
$|\xi+i|=O(1)$ as $\epsilon ~\rightarrow ~0$.
Applying Theorem 3.17 and transformations
(\ref{4.12}) and
(\ref{4.13}),
(\ref{2.35}) implies (\ref{4.11}).
For $\xi\in \mathcal{R}^-$, $\frac{\pi}{2}<\arg (\xi+i)<\pi$, which
on using
transformation (\ref{4.12}) implies
$0<\arg\eta<\frac{3}{4}\pi$. Continuity
implies that (\ref{4.11})
is satisfied for $\arg \eta =0$ as well.
Since $F'(\xi)\sim O(\epsilon)$,
and $\eta^{-2/3}\sim O(\epsilon^{4/3})$,
using (\ref{4.13}), we obtain (\ref{4.14}).
\end{proof}

\subsection{Leading Inner problem analysis}
\label{sec:4.2}

Setting $\epsilon=0$ in equation (\ref{4.11}),
we get the leading order
equation:
\begin{equation}
\label{4.15}
\frac{d\psi}{d\eta}+\psi=-\frac{1}{3\eta}-\frac{1}{3\eta}\psi \\
+\frac{\alpha}{6^{2/3}\eta^{2/3}}+\frac{1}{2}\sum_{n=2}^{\infty}(-1)^n(n+1)\psi^n
\end{equation}
with far-field matching condition:
\begin{equation}
\label{4.16}
\psi(\eta,a)\to 0 ~\qquad \text{as }|\eta|\to \infty, \quad 0\leq\arg \eta<\frac{3\pi}{4}.
\end{equation}
We shall prove the following theorem:
\begin{theorem}\label{thm:4.2}
There exists large enough $\rho_0 > 0$ such that (\ref{4.15}),
(\ref{4.16}) have a unique
analytic solution $\psi_0(\eta,\alpha)$ in the region $|\eta|\ge \rho_0,~
\arg\eta\in
(-\frac{\pi}{8},\frac{3\pi}{4})$.
\end{theorem}
The proof of this theorem will be given after some  definitions and lemmas.
\begin{definition}\label{def:4.3}
We define the region
$$\mathcal{R}_1=\{\eta:|\eta|> \rho_0,\quad \arg\eta\in (-\frac{\pi}{8},\frac{3\pi}{4})\}$$
for some large $\rho_0$ independent of $\epsilon$.
\end{definition}
\begin{definition}\label{def:4.4}
We define functions
\begin{equation}
\psi_1(\eta)=e^{-\eta},~\quad \psi_2(\eta)=e^{\eta}.
\end{equation}
\end{definition}
$\psi_1(\eta)$ satisfy the following equation exactly:
\begin{equation}
\mathcal{L}\psi\equiv\frac{d\psi}{d\eta}+\psi=0.
\end{equation}

Equation (\ref{4.15}) can be rewritten as
\begin{equation}
\label{4.19}
\mathcal{L}\psi =\mathcal{N}_1(\eta,\psi)
\equiv -\frac{1}{3\eta}-\frac{1}{3\eta}\psi \\
+\frac{\alpha}{6^{2/3}\eta^{2/3}}+\frac{1}{2}\sum_{n=2}^{\infty}(-1)^n(n+1)\psi^n.
\end{equation}
We consider solution $\psi $ of the following integral equation:
\begin{equation}
\label{4.20}
\psi=\mathcal{L}_1\psi
\equiv \psi_1(\eta)\int_{\infty e^{3\pi i/4}}^\eta \psi_2(t)\mathcal{N}_1(t,\psi(t))dt.
\end{equation}
\begin{definition}\label{def:4.5}
We define 
\begin{equation}
\begin{split}
\mathbf{B}_1&=\biggl\{\psi(\eta):\psi(\eta) \text{analytic in $\mathcal{R}_1$}\\
&\text{ and
continuous on $\overline{\mathcal{R}_1}$},
\underset{\mathcal{R}_1}{\sup}~|\eta^{2/3}\psi(\eta)|<\infty \biggr\}.
\end{split}
\end{equation}
\end{definition}
$\mathbf{B}_1$ is a Banach space with norm
\begin{equation}
\norm{\psi}=\underset{\mathcal{R}_1}{\sup}~|\eta^{2/3}\psi(\eta)|.
\end{equation}
\begin{lemma}\label{lem:4.6}
Let $\mathcal{N}\in \mathbf{B}_1$, then
$$\phi_1(\eta):=\psi_1(\eta)\int_{\infty e^{i 3\pi/4}}^{\eta}
\psi_2(t)\mathcal{N}dt\in \mathbf{B}_1,$$
and $\norm{\phi_1}\leq
K\norm{\mathcal{N}},$
where $K$ is independent of $\rho_0$.
\end{lemma}
\begin{proof}
For $\eta\in \mathcal{R}_1$, we use straight lines
in the $t$-plane to connect $\eta$
to $\infty e^{i 3\pi/4}$
so  $Re ~t$ is increasing monotonically
from $\infty e^{i 3 \pi/4}$ to $\eta$
and on that path, characterized by arc-length $s$,
$\frac{d}{ds} ~Re ~t(s) ~>~C ~>~0$. Further,
$ C_1|\eta|\leq |t|$ for nonzero $C_1$. Then,
\begin{equation*}
\begin{split}
|\phi_1(\eta)|&=\left\lvert\int_{\infty e^{i 3\pi/4}}^{\eta}
e^{\eta - t}\mathcal{N}dt\right\rvert\\
&\leq
~C
\int_{\infty e^{i 3\pi/4}}^{\eta}|t|^{-2/3} |e^{\eta - t}||t^{2/3}\mathcal{N}||dt|\\
&\leq
\frac{C\norm{\mathcal{N}}}{|\eta|^{2/3}}
\int_{\infty e^{i 3\pi/4}}^{\eta} |e^{\eta - t}||dt|
\leq \frac{C\norm{\mathcal{N}}}{|\eta|^{2/3}}.
\end{split}
\end{equation*}

\end{proof}
\begin{definition}\label{def:4.7}
Define $T_1(\eta)$ so that $T_1(\eta):= \mathcal{L}_1 0$.
\end{definition}
\begin{remark}
\label{rem:4.1}
Since $|\eta^{2/3} \mathcal{N}_1 (\eta, 0)|$ is bounded,
Lemma \ref{lem:4.6} implies
$T_1\in\mathbf{B}_1$.
\end{remark}
\begin{definition}\label{def:4.8}
We define
$\sigma_1=\norm{T_1}$,~~~
$\mathbf{B}_{\sigma_1}:= \{\psi\in\mathbf{B}_1: \norm{\psi}\leq 2\sigma_1\}.$
\end{definition}
\begin{lemma}\label{lem:4.9}
If $\psi\in\mathbf{B}_{\sigma_1},\phi\in\mathbf{B}_{\sigma_1}$,
then
$\mathcal{N}_1(\eta,\psi)\in \mathbf{B}_1$ and
$$\norm{\mathcal{N}_1(\eta,\psi)}~\leq ~2 K_1 \sigma_1 \left [\rho_0^{-2/3}
\sigma_1 + \rho_0^{-1} \right ] + \left [ \frac{1}{3 \rho_0}
+ \frac{2^{4/3} a}{3^{2/3}} \right ], $$
$$\norm{\mathcal{N}_1(\eta,\psi)
-\mathcal{N}_1(\eta,\phi)}~\leq ~K_1 \left [\rho_0^{-2/3}
\sigma_1 + \rho_0^{-1} \right ]
\left(\norm{\phi-\psi} \right) $$
for some numerical constant $K_1$ and for $8\sigma_1 \rho_0^{-2/3}< 1$.
\end{lemma}
\begin{proof}
It is clear from (\ref{4.19}) that
\begin{equation}
\label{4.18.5}
|\mathcal{N}_1 (\eta, \psi) - \mathcal{N}_1 (\eta, \phi) |
~\le ~\frac{|\psi (\eta) - \phi(\eta) |}{4 \|\eta \|^2}
+ \frac{1}{2} \sum_{k=2}^\infty (k+1) |\psi^k -\phi^k|.
\end{equation}
Noting,
$|\psi|\leq 2 \sigma_1 |\eta|^{-2/3}, ~|\phi|\leq 2 \sigma_1 |\eta|^{-2/3}$ and
from simple induction,
$$|\psi^k-\phi^k|\leq k(|\psi| + |\phi |)^{k-1}|\psi -\phi|,
\text{ for } k \ge 1,$$
we obtain for $\sigma_1 \rho_0^{-2/3} ~<~\frac{1}{8}$,
\begin{equation}
\label{4.18.8}
\| \mathcal{N}_1 (\eta, \psi) - \mathcal{N}_1 (\eta, \phi) \|
~<~ \left [ \frac{1}{3 \rho_0}
+ \frac{K_2 \sigma_1}{\rho_0^{2/3}} \right ]  ~\norm{\psi - \phi}
\end{equation}
for some numerical constant $K_2$.
On the other hand,
$$ \| \mathcal{N}_1 (\eta, 0) \| ~\le
~\frac{1}{3 \rho_0^{2/3}} +
\frac{2^{4/3} a}{3^{2/3}}. $$
So, it is clear from adding the two results above (with $\phi =0$), it
follows that
that for $\psi \in \mathbf{B}_{\sigma_1}$,
$$ \| \mathcal{N}_1 (\eta, \psi) \| ~<~
~\left [ \frac{1}{3 \rho_0^{2/3}}  +
\frac{2^{4/3} a}{3^{2/3}} \right ]
~+~\sigma_1 \left [ \frac{1}{4 \rho_0}
+ \frac{K_2 \sigma_1}{\rho_0^{2/3}} \right ].
$$
\end{proof}
\begin{lemma}\label{lem:4.10}
For sufficiently large $\rho_0$,
the operator $\mathcal{L}_1$ as defined in (\ref{4.20})
is a contraction from
$\mathbf{B}_{\sigma_1}$ to $\mathbf{B}_{\sigma_1}$;
hence there is a unique solution $\psi$ in this function space.
\end{lemma}
\begin{proof}
From Lemma \ref{lem:4.6} and Lemma \ref{lem:4.9}:
\begin{equation*}
\begin{split}
\norm{\mathcal{L}_1\psi -T_1}\leq \norm{\mathcal{L}_1\psi-\mathcal{L}_1 0}
&\leq 2 K \norm{\mathcal{N}_1(\eta,\psi)-\mathcal{N}_1(\eta,0)} \\
&\leq 2 K K_1 \left [ \rho_0^{-2/3} \sigma_1 + \rho^{-1} \right ]
\norm{\psi},
\end{split}
\end{equation*}
\begin{equation*}
\norm{\mathcal{L}_1\psi -\mathcal{L}_1\phi}\leq 2 K
\norm{\mathcal{N}_1(\eta,\psi)-\mathcal{N}_1(\eta,\phi)}\leq
2 K K_1 [\rho_0^{-2/3}\sigma_1 + \rho^{-1} ] \norm{\psi-\phi}.
\end{equation*}
So,
$$ \norm{\mathcal{L}_1 \psi} \le \norm{\mathcal{L}_1 \psi
-\mathcal{L}_1 0 } + \norm{T_1} \le 2 \sigma_1$$
for sufficiently large $\rho_0$.
\end{proof}
\begin{remark}
\label{rem:4.1.9}
It is easy to see that the previous lemma holds when
we change the restriction on $arg ~\eta$ in
the definition of $\mathcal{R}_1$ to $(0, ~\frac{3 \pi}{4} )$.
This comment is relevant to the following lemma.
\end{remark}
\begin{lemma}
\label{lem:4.14.5}
Any solution $\psi $ to (\ref{4.15})
satisfying condition (\ref{4.16}) in the domain
$\mathcal{R}_1$ must be in $\mathbf{B}_{\sigma_1}$
and satisfy integral equation:
$\psi = \mathcal{L}_1 \psi$  for sufficiently large $\rho_0$
\end{lemma}
\begin{proof}
First, we note that if we use variation of parameter, the
most general solution to (4.11) satisfies the integral
equation
$$ \psi = \mathcal{L}_1 \psi + C_1 \psi_1. $$
Now, if we assume $\| \psi \|_\infty$ to be small, as
implied by condition (\ref{4.16}), when $\rho_0$ is chosen
large, it follows from inspection of
the right hand side of (\ref{4.15}) that
$ \| \mathcal{N}_1 (\eta, \psi) \|_\infty$ is also small.
Since Lemma \ref{lem:4.6} is easily seen to hold when
the norm is replaced by $\| . \|_\infty$, it follows
that $\mathcal{L}_1 \psi$ is also small. However,
$C_1 \psi_1 (\eta) $ is unbounded in
$\mathcal{R}_1$ unless $C_1 = ~0$. Therefore,
any solution to (\ref{4.15}) satisfying
condition (\ref{4.16}) must satisfy integral equation
$\psi = \mathcal{L}_1 \psi$. If we
were to use the norm $\|. \|_\infty$ instead
of the weighted norm $\| . \|$ in the definition of
the Banach Space $\mathbf{B}_1$, it is easily seen that
each of Lemmas \ref{lem:4.6} - \ref{lem:4.10} would remain valid
for small enough $\sigma_1$, as appropriate when condition
(\ref{4.16}) holds and $\rho_0$ is large.
Thus, it can be concluded that
the solution to $\psi = \mathcal{L}_1 \psi$
is unique in the bigger space of functions for
which $\psi$ satisfies (\ref{4.16}) and
$\rho_0$ is chosen large enough.
However, from previous Lemma \ref{lem:4.10}, it follows
that this unique solution must be in the function space
$\mathbf{B}_1$ and therefore satisfies
$\psi ~=~O(\eta^{-2/3})$  for large $\eta$.
\end{proof}

\noindent{\bf Proof of Theorem \ref{thm:4.2}}.
This follows immediately from Lemmas \ref{lem:4.10} and \ref{lem:4.14.5}.

\begin{theorem}\label{thm:4.11}
If $\psi_0(\eta,\alpha)$ is the solution in Theorem \ref{thm:4.2}, then \\
$\mathrm{Im}~ \psi_0(\eta,\alpha)= S(\alpha)
e^{-\eta}\left(1+o(1)\right)$ on the real $\eta$ axis and $\eta\to\infty$.
\end{theorem}
\begin{proof}
Plugging $\psi_0 =\text{ Re ~}\psi_0 +i\text{ Im ~}\psi_0$ in equation (\ref{4.15}),
then taking imaginary part, we find that
Im $\psi_0(\eta)$ satisfies the following linear homogeneous equation on the real positive $\eta$ axis:
\begin{equation}
\label{4.22}
\frac{d~\text{Im}~\psi_0}{d\eta}+(1+\frac{1}{3\eta}+E(\eta))~\text{Im}~\psi_0 =0,
\end{equation}
where $E(\eta)$ is obtained from an homogeneous expression of Re~ $\psi_0$ and Im~ $\psi_0$.
Since {\it apriori} both
$\text{Re~ }\psi_0\sim O(\eta^{-2/3})$ and Im $~\psi_0\sim O(\eta^{-2/3}) $
as $\eta\to\infty$, we obtain
\begin{equation}
E(\eta)\sim O(\eta^{-2/3}),\quad \text{ as } \eta\to\infty.
\end{equation}
From Theorem 6.2.1 in Olver \cite{Olver}, there is a solution $\tilde \phi_0(\eta)=e^{-\eta}(1+o(1))$ of (\ref{4.22}).
 Hence there is a constant  $S(\alpha)$ so that
Im~ $\psi_0(\eta)=S(\alpha)e^{-\eta}(1+o(1))$.

\end{proof}
\begin{remark}
\label{rem:4.2}
From Theorem \ref{thm:4.11},
Im~ $\psi_0(\eta,\alpha)=0$ iff $ S(\alpha)=0$.
Previous numerical results and formal
asymptotic results ( Chapman and King \cite{Chapman2})
suggest that $S(\alpha)=0$ if and only if $\alpha$ takes on a
discrete set of values.
\end{remark}

\subsection{Full Inner Problem Analysis}
\label{sec:4.3}

Now we go back to the full inner equation (\ref{4.11}).
From Theorem \ref{thm:4.1}, (\ref{4.11})
with matching condition (\ref{4.14})
has unique solution in the domain
$k_1\epsilon^{-2}\ge |\eta|\geq k_0\epsilon^{-2},\quad \arg\eta
\in (0,\frac{3\pi}{4})$.
We shall first prove that this solution can be extended to the
region: \\
$\mathcal{R}_{2}=\{ \eta:\rho_0 < \text{Im~ }\eta + \text{Re~ }\eta <\tilde k_0\epsilon^{-2},\quad \arg \eta\in [0,\frac{3\pi}{4});\quad -\text{Im~ }\eta +\rho_0 < \text{Re~ }\eta < \text{Im~ }\eta +\tilde k_0\epsilon^{-2},\quad \arg \eta\in (-\frac{\pi}{8},0]\}$,
where $k_0 <\tilde k_0< k_1$.
\begin{definition}\label{def:4.12}
Let $\psi = {\tilde \psi} (\eta)$ be the unique analytic solution in
Theorem \ref{thm:4.1}
for $|\eta|\geq k_0\epsilon^{-2
},\quad \arg\eta\in (0,\frac{3\pi}{4})$, restricted
to the line segment $\{\eta: \text{Im~ }\eta + \text{Re~ }\eta=\tilde k_0\epsilon^{-2},
\quad \arg\eta\in [0,\frac{3\pi}{4})\}$.
\end{definition}
\begin{definition}\label{def:4.13}
We define
\begin{equation}
\eta_0=\tilde k_0\epsilon^{-2}, ~~\quad \eta_1=\tilde k_0\epsilon^{-2}
\frac{1}{\sin\frac{\pi}{4}-1}e^{\frac{3i\pi}{4}},~~\quad \eta_2=i\tilde k_0\epsilon^{-2}.
\end{equation}
\end{definition}

Equation (\ref{4.11}) can be rewritten as
\begin{equation}
\label{4.26}
\begin{split}
\mathcal{L}\psi &=\mathcal{N}_2(\eta,\epsilon,\psi)\\
&\equiv -\frac{1}{3\eta}-\frac{1}{3\eta}\psi
+\frac{\alpha}{6^{2/3}\eta^{2/3}}\\
&+\frac{1}{2}\sum_{n=2}^{\infty}(-1)^n(n+1)\psi^n+\epsilon^{2/3}
E_3(\epsilon^{2/3},\epsilon^{2/3}\eta^{2/3},\psi,\eta^{-2/3}).
\end{split}
\end{equation}
We consider solution $\psi $ of the following integral equation:
\begin{equation}
\label{4.27}
\psi=\mathcal{L}_2\psi
\equiv \psi_1(\eta)\int_{\eta_1}^\eta \psi_2(t)
\mathcal{N}_2(t,\psi)dt+\tilde\psi(\eta_1).
\end{equation}

\begin{definition}\label{def:4.14}
We define
\begin{equation}
\mathbf{B}_2=\{\psi(\eta):\psi(\eta)
~\text{analytic in }\mathcal{R}_2,\text{ and continuous on } \overline{\mathcal{R}}_2~\}.
\end{equation}
$\mathbf{B}_2$ is a Banach space with norm
\begin{equation}
\norm{\psi}=\underset{\overline{\mathcal{R}_2}}{\sup}~\rho_1|\psi(\eta)|,\text{ where } \rho_1=\rho_0^{2/3}.
\end{equation}
\end{definition}
\begin{lemma}\label{lem:4.15}
Let $\mathcal{N}\in \mathbf{B}_2$. Then
$$\phi_1(\eta):=\psi_1(\eta)
\int_{\eta_1}^{\eta} \psi_2(t)
\mathcal{N} (t) dt\in \mathbf{B}_2;$$
and $\norm{\phi_1}\leq
K \norm{\mathcal{N}} $, where $K$ is a numerical
constant independent of any parameters.
\end{lemma}
\begin{proof}
For $\eta\in \mathcal{R}_2$,
we use straight lines to connect$\eta$ to $\eta_1$
so that Re $t$ is increasing from $\eta_1$ to $\eta$ and $ C_1|\eta|\leq |t|\leq C_2|\eta|$. Then
\begin{equation*}
\begin{split}
|\phi_1(\eta)|&=\left\lvert\int_{\eta_1}^{\eta}
e^{-\eta + t}\mathcal{N}dt\right\rvert\\
&\leq \frac{K}{\rho_0^{2/3}}\int_{\eta_1}^{\eta} |e^{-\eta + t}| |
\rho_0^{2/3}
\mathcal{N}||dt|\\
&\leq \frac{K\norm{\mathcal{N}}}{\rho_0^{2/3}}\int_{\eta_1}^{\eta}
|e^{\eta - t}||dt|\leq \frac{K \norm{\mathcal{N}}}{\rho_0^{2/3}}.
\end{split}
\end{equation*}
\end{proof}
\begin{definition}
We define
$T_2(\eta): =\mathcal{L}_2 0$,~~~
$\sigma_2 \equiv \norm{T_2}.$
\end{definition}
\begin{remark}
\label{rem:4.2.5}
Since $\mathcal{N}_2 (\eta, \epsilon, 0, 0) =
O(\eta^{-4/3}, \epsilon^{2/3})$
in $\mathcal{R}_2$, it follows from Lemma \ref{lem:4.15} that
$\sigma_2 = O(1)$, as
$\epsilon ~\rightarrow~0^+$.
\end{remark}
We define space
$$\mathbf{B}_{\sigma_2}=\{\psi\in\mathbf{B}_2:\norm{\psi}\leq 2\sigma_2\}.$$
\begin{lemma}
\label{lem:4.16}
If $\psi\in\mathbf{B}_{\sigma_2}$, $ \phi\in\mathbf{B}_{\sigma_2}$,
then for $\rho_1 > \max~ \{ 4 \sigma_2 \rho_2, 8 \sigma_2 \}$, $\rho_2$
as defined in equation (\ref{4.11.8}),
\begin{equation*}
\begin{split}
\norm{\mathcal{N}_2(\eta,\epsilon,\psi)
-\mathcal{N}_2(\eta,\epsilon, \phi)}
\leq
&\left [ \frac{1}{3 \rho_0} + \frac{K_2 \sigma_1}{\rho_0^{2/3}}
+ 8 A \sigma_2 \epsilon^{2/3} \right ] \norm{\phi-\psi}
\end{split}
\end{equation*}
for sufficiently small $\epsilon$, where $K_2$ is a numerical constant.
\end{lemma}
\begin{proof}
Using (\ref{4.26}),
$|\psi|\leq 2\sigma_2 \rho_1^{-1},~|\phi|\leq 2\sigma_2 \rho_1^{-1}$
and inequality
$$|\psi^k-\phi^k|\leq k(|\psi|+|\phi|)^{k-1}|\psi -\phi|,\text{ for }
k \geq 2$$
we have from (\ref{4.11.8})
\begin{equation*}
\begin{split}
&\epsilon^{2/3} | E_3(\epsilon^{2/3}, \epsilon^{2/3} \eta^{2/3},
\psi, \eta^{-2/3} )-E_3 (\epsilon^{2/3}, \epsilon^{2/3} \eta^{2/3},
\phi, \eta^{-2/3} ) | \\
&\leq \frac{8 A \rho_2 \epsilon^{2/3}}{\rho_1} \norm{\phi - \psi}.
\end{split}
\end{equation*}
Combining with (\ref{4.18.8}), the lemma follows.
\end{proof}

\begin{theorem}\label{thm:4.18}
There exists a unique solution $\psi\in\mathbf{B}_{\sigma_2}$ of equation
(\ref{4.27}) for all sufficiently large $\rho_0$ and small $\epsilon$.
\end{theorem}
\begin{proof}
Using (\ref{4.27}), Lemma \ref{lem:4.15} and Lemma \ref{lem:4.16},
it is easily seen that
\begin{equation*}
\begin{split}
\norm{\mathcal{L}_2 (u) }
&\le \norm{\mathcal{L}_2 (u) - \mathcal{L}_2 (0) } + \norm{T}_2 \\
&\le \sigma_2 + 4 \sigma_2 K \left [ \frac{1}{3 \rho} +
\frac{\sigma_1 K_2}{\rho_0^{2/3}} + 8 A \sigma_2 \epsilon^{2/3} \right ]
< 2 \sigma_2
\end{split}
\end{equation*}
for sufficiently large $\rho_0$ and small $\epsilon$.

On the other hand,
$$\norm{\mathcal{L}_2 (u_1) -\mathcal{L}_2 (u_2) }
\le K \left [ \frac{1}{3 \rho_0} +
\frac{\sigma_1 K_2}{\rho_0^{2/3}} + 8 A \sigma_2 \epsilon^{2/3} \right ]
\norm{(u_1-u_2)}. $$
Hence the
proof follows from contraction mapping theorem on a Banach space.
\end{proof}
\begin{lemma}\label{lem:4.19}
Let $\psi(\eta)$ be the solution of (\ref{4.27})
as in Theorem \ref{thm:4.18} then $\psi(\eta)$ is a solution of (\ref{4.26})
with $\psi(t)\equiv \tilde\psi(t)$ for $t\in\{\eta : ~\text{Re ~}\eta +~\text{Im~ }\eta =\tilde k_0\epsilon^{-2},~\arg \eta\in[0,\frac{3\pi}{4}]\}$.
\end{lemma}
\begin{proof}
Since $\tilde\psi(t)$ is a solution of (\ref{4.26}) for $t\in\{\eta :
~\text{Re ~}\eta +~\text{Im~ }\eta =\tilde k_0\epsilon^{-2},~\arg \eta\in[0,\frac{3\pi}{4}]\}$,
by variation of parameters:

\begin{equation}
\label{4.29}
\tilde\psi(t)=\mathcal{L}_2\tilde\psi+\tilde\psi(\eta_1).
\end{equation}

By equation (\ref{4.27}), we have
\begin{equation}
\begin{split}
\psi(t)-\tilde\psi(t)&=\psi_1(t)\int_{\eta_1}^{t}\psi_2(t)\left(\mathcal{N}_2(\psi)-\mathcal{N}_2(\tilde\psi)\right)dt.
\end{split}
\end{equation}
Using Lemma \ref{lem:4.15} and Lemma \ref{lem:4.16} , we have
$$\norm{\psi-\tilde\psi}\leq
K \left [ \frac{1}{3 \rho^1} +
\frac{\sigma_1 K_2}{\rho_0^{2/3}} + 8 A \sigma_2 \epsilon^{2/3} \right ]
\norm{(\psi-\tilde\psi)}$$
for sufficiently large $\rho_0$ and small $\epsilon_0$.
So $\psi(t)\equiv \tilde\psi(t)$.
\end{proof}
\begin{theorem}\label{thm:4.20}
For large enough $\rho_0$,
there exists a unique solution \\
$\psi(\eta,\epsilon,\alpha)$ of (4.7),
(\ref{4.14}) in region $ R_1\ge |\eta| \geq \rho_0$ for $\arg \eta\in [0,3\pi /4]$, where $R_1$ is some constant chosen to be independent of $\epsilon$. Furthermore the solution $\psi(\rho_0,\epsilon,\alpha)$ 
 is analytic in $\epsilon^{2/3}$ and $\alpha$, as $\epsilon\to 0$ and satisfies $\lim_{\epsilon \rightarrow 0^+}
\psi(\eta,\epsilon,\alpha)=\psi_0(\eta,\alpha)$
for $R_1\ge |\eta| \ge \rho_0$.
\end{theorem}
\begin{proof}
The first part  follows from Theorem \ref{thm:4.1} and
Theorem \ref{thm:4.18} and Theorem \ref{lem:4.19} .
Note that as $\epsilon\to 0$,
$\epsilon^{2/3}E_3(\epsilon,\eta,\psi)\to 0$
uniformly in the region given.
The second part  follows from the theorem of dependence
of solution on parameters (see, for instance, Theorem 3.8.5 in Hille \cite{7}).
\end{proof}
\begin{lemma}\label{lem:4.21}
Let $F(\xi)$ be the solution of the Half Problem in Theorem 3.17, we define $q(\xi)$
so that $q(\xi)=\frac{F(\xi)-[F(-\xi^*)]^*}{2i}$
(Note this is the same as $\mathrm{Im} ~F$ on
$\{~\mathrm{ Re } ~ \xi = 0 \}\cap\mathcal{R}$).
Then $q$ satisfies the following
homogeneous equation on the imaginary $\xi$ axis: $\{\xi=is\}$:
\begin{equation}
\label{4.31}
\epsilon^2\frac{dq}{ds}+(2iH(is)Q(is) +\tilde L(s))q=0,
\end {equation}
where $\tilde L(s)$ is some real function and $\tilde L(s)\sim O(\epsilon^2)$ as $\epsilon\to 0$.
\end{lemma}
\begin{proof}
On the imaginary axis $\{\xi=is\}$:
Using (\ref{2.11}),
$L(is)=\frac{\sqrt{\gamma^2-s^2}(s+\gamma)}{(1-s^2)^2}$ is real.
Using (\ref{2.58}), $i\left(\bar H-H\right)(is)
=\frac{2\gamma}{(1-s^2)}$ is real.
By taking imaginary part in equation (\ref{2.42}),
we have the lemma.
\end{proof}
\begin{remark}
\label{rem:4.3}
$F$ is analytic in $[-i+i\epsilon ^{2/3}\rho_0,-ib)$.
\end{remark}
\begin{lemma}\label{lem:4.22}
If $q(is_1)=0$ with $-1+\epsilon^{4/3}\rho_0\leq s_1<-b$,
then $q(is)\equiv 0$ for all $s\in [-i+\epsilon^{4/3}\rho_0,0]$.
Conversely, if $q(s_1) \ne 0$ for
$s_1 \in [-1+\epsilon^{4/3} \rho_0, -b)$, then $F$ cannot
satisfy symmetry condition: $\mathrm{ Im }~F=0$ on
$\{ ~\mathrm{ Re ~}\xi = 0\} \cap \mathcal{R}$.
\end{lemma}
\begin{proof}
By equation (\ref{4.31}),
\begin{equation*}
q(is)=c_1\exp \{\frac{1}{\epsilon}\int_{s_1}^{s}L(it)dt\}\{1+o(1)\};
\end{equation*}
hence
if $q(is_1)=0$, then $c_1=0$. Conversely, if
$q(i s_1) ~\ne ~0$, then $c_1\ne 0$. Hence $F$ cannot satisfy the symmetry condition
Im ~$F =0 $ on $\{~\text{Re ~}\xi = 0 \} \cap \mathcal{R}$.
\end{proof}
\begin{theorem}\label{thm:4.23}
Assume $S(\beta_0) = 0$, but $S'(\beta_0)\neq 0$,
then for small enough $\epsilon$ and large enough
$\rho_0$, there is analytic function
$\beta(\epsilon^{2/3})$ such that\\
$\lim_{\epsilon\to 0}\beta(\epsilon^{2/3})=\beta_0$,
and if $\lambda$ satisfies (\ref{1.28}),
then $\mathrm{Im }~ F(\xi)=0$ on $\{\xi =is \}\cap \mathcal{R}$.
\end{theorem}
\begin{proof}
For fixed large enough $\rho_0>0$, $S'(\beta_0)\neq 0$ implies
\begin{equation}\label{4.42}
\frac{\partial\text{ Im }~\psi_0}{\partial \alpha}(\rho_0, \beta_0)\neq 0,
\end{equation}
Using Theorem \ref{thm:4.20}, (\ref{4.42}) and the
implicit function theorem, there exists analytic function
$\beta(\epsilon^{2/3})$
so that Im ~$\psi(\rho_0,\epsilon,\beta(\epsilon^{2/3}))=0$.
This implies that $q(is)$ is zero at some point in
$[-i+i \rho_0^{4/3}, -bi)$. Then using Lemma \ref{lem:4.22},
we complete the proof.
\end{proof}

{\bf Proof of Theorem \ref{thm:1.9}}: If $\lambda$ satisfies restriction (1.31),
from Theorem \ref{thm:4.23} and Theorem 2.39, $F(\xi)$
is a solution of the Finger Problem.

\section{Conclusion and Discussion}

In this paper we are concerned with the existence and selection of steadily translating
symmetric finger solutions in a Hele-Shaw cell by small but non-zero
kinetic undercooling $\epsilon^2$.
We rigorously conclude that for relative finger width
$\lambda$ in the range
$[\frac{1}{2}, \lambda_m ] $, with $\lambda_m - \frac{1}{2}$ small,
symmetric finger solutions exist in the asymptotic limit
of undercooling $\epsilon^2 ~\rightarrow ~0$ if
the Stokes multiplier for a relatively simple nonlinear differential
equation is zero. This Stokes multiplier $S$
depends on the parameter
$\alpha \equiv \frac{2 \lambda -1}{(1-\lambda)}\epsilon^{-\frac{4}{3}} $ and
earlier calculations \cite{Chapman2}  have shown
this to be zero for  a discrete set of values of $\alpha$. While this result is similar to that obtained in \cite{Xie1} previously for the Saffman-Taylor fingers by surface tension, the analysis for the problem with kinetic undercooling exhibits a number of subtleties as pointed out by Chapman and King \cite{Chapman2}. The main subtlety is the behavior of the Stokes lines at the finger tip, where the analysis is complicated by non-analyticity of the coefficients in the governing equation. \\

A recent study by Dallaston and McCue \cite{Dallaston} showed that for a given kinetic undercooling parameter $\epsilon^2$, a continuous family of corner-free finger solutions exist with width $\lambda\in (\lambda_{\min},1)$, where $\lambda_{\min}\to 0$ as $\epsilon\to 0$. This result did not need to contradict with the selection result that was obtained asymptotically in Chapman and King \cite{Chapman2} and was confirmed in this paper, since the numerical scheme in \cite{Dallaston} can not distinguish between solutions with analytic fingers and those with fingers that are corner-free but may not be analytic at the nose. However results in \cite{Xie3} showed that for sufficiently small $\epsilon$, no symmetric analytic finger solution exists for $\lambda<\frac{1}{2}$. The methods developed in this paper suggest that any corner-free symmetric solution satisfying some non-degenerate condition at the finger tip must be analytic at the finger nose. We will elaborate about this in a forthcoming paper.\\

More recently, Gardiner et al \cite{Gardiner} have constructed numerical solutions to the finger problem with undercooling that are analytic. Their strategy is to add surface tension $\sigma$ to the model, so that (\ref{1.6}) is replaced by $\phi=\sigma\kappa +cv_n$, where $\kappa$ is the curvature of the finger. A key hypothesis in \cite{Gardiner} is that the work of Tanveer and Xie \cite{Xie1,Xie2} carries over to the new model so that solutions to the problem with kinetic undercooling and surface tension must be analytic at the nose. Thus with kinetic undercooling fixed at some value, by taking the limit $\sigma\to 0$, one selects the analytic solutions studied in \cite{Chapman2} and in this paper. We tried to verify rigorously this strategy and hypothesis in \cite{Gardiner}, but we were not able to do so. The Stokes line picture at the finger tail and nose for the problem with kinetic undercooling and surface tension is quite different from that for the problem with only surface tension, thus one can not use the same methods as in \cite{Xie1,Xie2} to control an integral operator that is related to the surface tension $\sigma$.\\

For the time dependent finger problem with kinetic undercooling, local existence of analytic solution was obtained in \cite{24,25} for analytic initial data. Dallaston and McCue \cite{Dallaston} numerically demonstrated corner formation for sufficiently high kinetic undercooling in finite time. Stability analysis in Chapman and King \cite{Chapman2} showed that all wave numbers are unstable. Thus it seems impossible for a global solution to exist. We believe that the techniques developed in this paper are also useful in studying the linear stability of the Saffman-Taylor fingers with symmetric or antisymmetric disturbances.

\section{Appendix A: Properties of the function $P(\xi)$}

We look into some properties of $P(\xi)$ which is defined by (\ref{2.36}) and (\ref{2.47}).

\begin{lemma}
 $\mathrm{Re}~ P(\xi)$ increases
along negative $\mathrm{Re}~ \xi$ axis $(-\infty,0)$
with $\mathrm{Re}~ P(-\infty)=constant\ne 0$ and $C_1|t-2i|^{-2}\leq |\frac{d}{dt}\mathrm{ Re }~P(t)|\leq \frac{C_2}{\nu} |t-2i|^{-2}$, for $t\in (-\infty,-\nu)$
where $C_1$ and $C_2$ are positive constants, independent
of $\epsilon$ and $\nu$.
\end{lemma}
\begin{proof}
Using (\ref{2.36}) and (\ref{2.47}), we obtain
\begin{equation*}
\text{Re }P(\xi)=-\gamma\int^\xi_{-\nu} \frac{(s^2+\gamma^2)^{1/2}}{s(s^2+1)}ds + constant.
\end{equation*}
The lemma follows from the above equation.
\end{proof}
\begin{lemma}
$\mathrm{Re}~ P(\xi)$ increases monotonically on the imaginary
$\xi$ axis from $-ib$ to $0$ where $0< b <\min\{1,\gamma\}$.
\end{lemma}
\begin{proof}
Using (\ref{2.36}) and (\ref{2.47}), we obtain for $\xi=i\eta$
\begin{equation*}
\text{Re }P(\xi)=-\int_{-b}^\eta\frac{(\gamma +t)(\gamma^2-t^2)^{1/2}}{t
(1-t^2)} + constant.
\end{equation*}
The lemma follows from the above equation.
\end{proof}

\begin{lemma} There exists a constant $R$
independent of $\epsilon$ so that for $|\xi|\geq R$,
$\mathrm{Re}~ P(t)$ increases with decreasing $s$ along any ray
$r=\{t:t=\xi-se^{i\varphi},~0<s<\infty,~0 \le \varphi  <
\frac{\pi}{2} \}$
in $\mathcal{R}$ from $\xi$ to
$\xi-\infty e^{i\varphi}$ and
$C_1|t-2i|^{-2}\leq |\frac{d}{ds}~\mathrm{ Re }~P(t(s)|\leq C_2 |t-2i|^{-2}$,
where $C_1$ and $C_2$ are constants, independent
of $\epsilon$, with $C_1 ~>~0$.
\end{lemma}
\begin{proof}
Using (\ref{2.36}) and (\ref{2.47}), we obtain for $|\xi|\to \infty$, $P(\xi)\sim -i\log \xi$ ; hence the lemma follows.
\end{proof}

\begin{corollary}
On line segment $r_{u_1}$ with decreasing $s$, $\mathrm{Re }~P$ increases  with
$\frac{d}{ds} ~\mathrm{Re} ~P(t(s)) ~>~\frac{C}{|t(s)-2i|^2} ~>~0$ for constant
$C$ independent of $\epsilon$ and $\lambda$ when the latter is restricted
to a compact subset of (0, 1).
\end{corollary}
\begin{proof}
This follows very simply from the previous lemma.
\end{proof}

\begin{lemma}
There exists a constant $R$
independent of $\epsilon$ so that for $|\xi|\geq R$,
$\mathrm{Re}~ P(t)$ increases with increasing $s$ along any arc
$r=\{t:t=|\xi|e^{is},~\pi/2<s<3\pi/2\}$
in $\mathcal{R}$ .
\end{lemma}
\begin{proof}
Using (\ref{2.36}) and (\ref{2.47}), we obtain for $|\xi|\to \infty$, $P(\xi)\sim -i\log \xi$ ; hence the lemma follows.
\end{proof}

\begin{lemma}
\label{lem:A.5}
For $\lambda$ in a compact subset of (0, 1) and for  $R$ independent of $\epsilon$,  consider the line segment
$$ t = i \nu_1 - \nu_1 - s e^{i \phi} ~~,~~0 \le s \le R$$
with $\nu_1~>0$.
Then, there exists real $\nu_1$, $\phi$ sufficiently small
in absolute value and depending only on  $R$ so that on this
line segment
$$ -\frac{d}{ds} ~\mathrm{Re} ~P(t(s)) ~>~C ~>~0 ,$$
where $C$ is independent of $\epsilon$.
\end{lemma}
\begin{proof}
Note for $\nu_1=0$ and $\phi=0$, result holds from Lemma 5.1,
with $C=C_1$ only depending on the lower bound for $\nu$.
Since $\frac{d}{ds} ~\text{Re} ~P (t(s)) $ is clearly a continuous function
of $\phi$ and $\nu_1$, and uniformly continuous for
for $s$ restricted in a compact set, it follows that there exists
$\phi$ and $\nu_1$ small enough that
$$ -\frac{d}{ds} ~\text{Re} ~P(t(s)) ~>~\frac{C_1}{2} = C ~>~0 ,$$
where $C$ is only dependent on $\nu_0$
\end{proof}
\begin{corollary}
\label{cor:A.8}
For small enough  $\nu_1 ~>~0$,
on line segment $r_{u,2}$, parameterized by arclength $s$ increasing
towards $R$,
$$-\frac{d}{ds} ~\mathrm{Re} ~P (t (s) ) ~>~C ~>~0 ,$$
where constant $C$ is independent of $\epsilon$ and $\lambda$.
\end{corollary}
\begin{proof}
For $r_{u_2}$, we use previous Lemma 5.6
with  $\phi=0$
to obtain desired result.
\end{proof}

\begin{lemma} There exists sufficiently small $\nu_1 > 0$ independent of $\epsilon$  so that $\frac{d}{ds}~ \mathrm{Re}~ P(t(s))\geq C >0$ on the parameterized straight lines $\{t(s)=-\nu_1 +se^{-i\frac{\pi}{6}},0\leq s\leq 2\sqrt{3}\nu_1/3\}$ and $\{t(s)=-\nu_1 +se^{i\frac{\pi}{6}},0\leq s\leq 2\sqrt{3}\nu_1/3\}$,  $C$ is some constant independent of $\epsilon$ and $\nu_1$.
\end{lemma}
\begin{proof}
Note that $\tilde Q(\xi)\sim -\frac{\gamma^2}{\xi}$ near $\xi=0$, hence $P(\xi)\sim -\gamma^2 \log \xi$ which implies the lemma.
\end{proof}
\begin{lemma}
\label{lem:A.7}
There exist $~0<b<1$ and $0<\alpha_0<\pi/2$ so that $\mathrm{ Re }~P(t)$ is decreasing along ray $r_l=\{\xi=-bi+se^{i(\pi+\alpha_0)},0\leq s<\infty\}$.
\end{lemma}
\begin{proof}
We want to show that
\begin{equation}
\label{A.8}
\frac{d}{ds}\text{ Re }P(\xi(s))=\text{ Re }\{P'(\xi)e^{i(\pi +\alpha_0)}\}<0 \quad \text{ on }r_l.
\end{equation}
Note:
\begin{equation}
P'(\xi)=i\frac{(\xi+i\gamma)}{\xi}\frac{\left((-i\gamma+\xi)(\gamma+\xi)\right)^{1/2}}{(\xi+i)(\xi-i)},
\end{equation}
so
\begin{multline}
\arg \left(P'(\xi)e^{i(\pi +\alpha_0}\right)=\frac{3\pi}{2}+\alpha_0-\left(\arg \xi(s)-\arg(\xi+i\gamma)\right)\\
-[\left(\arg(\xi+i)+\arg (\xi-i)\right)-\frac{1}{2}\left(\arg(\xi+i)+\arg (\xi-i)\right)]\in [\pi,\frac{3\pi}{2}-\frac{\alpha_0}{2}]
\end{multline}
which leads to (\ref{A.8}).
\end{proof}

\noindent {\bf Acknowledgment:} The author thanks Professor Saleh Tanveer of The Ohio State University for suggesting the problem and helpful discussions. The author also thanks the associate editor and the reviewers for their very helpful comments.

\end{document}